\documentclass[arxiv,reqno,twoside,a4paper,12pt]{amsart}

\usepackage[utf8]{inputenc}
\usepackage{amsmath}
\usepackage{amsfonts}
\usepackage{amssymb}
\usepackage{amsthm}
\usepackage{float}
\usepackage{mathtools}
\usepackage{csquotes}
\usepackage[numbers]{natbib}
\usepackage{bbold}
\usepackage{esint}
\usepackage[colorlinks=true,linkcolor=blue,citecolor=blue]{hyperref}

\usepackage[scaled]{helvet} 
\usepackage{courier} 
\usepackage[mathbf]{euler}
\usepackage[dvipsnames]{xcolor}

\usepackage[driver=pdftex,margin=3cm,heightrounded=true,centering]{geometry}
\usepackage[arrow, matrix, curve]{xy}
\usepackage{tikz-cd}
\usetikzlibrary{arrows,matrix,shapes,decorations.text, mindmap,shapes.misc}
\usepackage{metalogo}

\newcommand{\N}{\mathbb{N}}
\newcommand{\C}{\mathbb{C}}
\newcommand{\Z}{\mathbb{Z}}

\newcommand{\T}{\mathbb{T}}
\renewcommand{\S}{\mathbb{S}}

\newcommand{\R}{\mathbb{R}}

\newcommand{\abs}[1]{\left\lvert#1\right\rvert}

\newcommand{\Tr}[0]{\textup{Tr}}

\usepackage[numbers]{natbib}

\newcommand{\id}{\mathrm{id}}

\newcommand{\rint}{\fint}
\newcommand{\Rel}{\mathrm{Re}}

\tikzset{cross/.style={cross out, draw=black, minimum size=2*(#1-\pgflinewidth), inner sep=0pt, outer sep=0pt},
default radius will be 1pt. 
cross/.default={1pt}}

\theoremstyle{plain}
\newtheorem{Thm}{Theorem}[section]
\newtheorem{Conj}{Conjecture}[section]
\newtheorem{Hyp}{Conjecture}[section]
\newtheorem{Ans}[Thm]{Answer}
\newtheorem{Lem}[Thm]{Lemma}

\newtheorem{Prop}[Thm]{Proposition}
\newtheorem{Def}[Thm]{Definition}
\newtheorem{Rem}[Thm]{Remark}

\theoremstyle{plain}

\numberwithin{equation}{section}

\begin{document}

\title{Spectral zeta function on discrete tori and Epstein-Riemann conjecture}

\author{Alexander Meiners \& Boris Vertman}
\address{Mathematisches Institut,
Universit\"at Oldenburg,
26129 Oldenburg,
Germany}
\email{alexander.meiners@uni-oldenburg.de}
\email{boris.vertman@uni-oldenburg.de}


\begin{abstract}
We consider the combinatorial Laplacian on a sequence of discrete tori which approximate the $\alpha$-dimensional torus. In the special case $\alpha=1$, 
Friedli and Karlsson derived an asymptotic expansion of the corresponding spectral zeta function in the critical strip, as the approximation parameter goes to infinity. 
There, the authors have also formulated a conjecture on this asymptotics, that is equivalent to the Riemann conjecture. 
In this paper, inspired by the work of Friedli and Karlsson, we 
prove that a similar asymptotic expansion holds for $\alpha=2$. Similar argument applies to higher dimensions as well. A conjecture on this asymptotics
gives an equivalent formulation of the Epstein-Riemann conjecture, if we replace the standard discrete Laplacian with the $9$-point star 
discrete Laplacian.
\end{abstract}

\maketitle

\tableofcontents

\section{Introduction and statement of the main results}\label{intro-section}

In this paper we focus on two main objectives. First, we study asymptotics of the 
spectral zeta function on finite torus graphs as they grow to infinity, or equivalently approximate the $m$-dimensional torus after rescaling.
Such sequences of graphs play an important role in mathematical physics and statistical mechanics, see e.g. \cite{DuDa88, Lo, RV} to name a few.
\medskip

Second objective is to provide a new perspective to the higher-dimensional analogue of 
the Riemann conjecture $-$ the Epstein-Riemann conjecture, see \cite{Eps}.  
Hereby, one intriguing feature is that our reformulation of the Epstein-Riemann conjecture requires a refinement of the standard discrete Laplacian 
$-$ the $9$-point star operator that is widely used in numerical analysis, see e.g. \cite{Bra}. \medskip

Our paper is a generalization of the previous work by \cite{FrKa}. In fact, aspects of spectral analysis
on discrete graphs has been the focus of intensive research by several groups. We shall name a few, without attempting to provide a 
complete list. \medskip

In the case of rectilinear polygonal domains, \cite{Ken}
derived a partial asymptotic expansion for the determinant of combinatorial Laplacian. A similar problem in a different
discrete setting of half-translation surfaces endowed with a unitary flat vector bundle, has been studied recently in \cite{Finski}.
In the special case of a discrete torus, \cite{CJK} as well as the second named author in \cite{Ver} identified the constant term in that 
expansion in terms of the zeta-regularized determinant of the Laplace-Beltrami operator. This relation to the 
zeta-regularized determinant was shown to be true in a much more general setting in \cite{Iz}. 
Let us also mention \cite{Sridar}, which studied the asymptotic
determinant for variations of the Riemannian metric, and a recent result in
\cite{TrSa}, where this problem was discussed on a symmetric discretization of surfaces glued together by a 
finite number of equilateral triangles. In this case, the constant term is given by the zeta regularized determinant plus some lattice depending summands.
\medskip

While the references above are concerned with the asymptotic behaviour of the discrete determinant,
we study asymptotics of the discrete zeta function in the critical strip and its relation to the (Epstein-) Riemann 
conjecture. In that respect, \cite{FrKa} and our current work here, stand alone. Our main results 
are in Theorems \ref{t-asymp} and \ref{ER-equiv}. Informally, they can be formulated as follows.
  
\begin{Thm}\label{main}
		Consider the $9$-point star combinatorial Laplacian $\widetilde{\Delta}_n$ on a finite torus
		graph that approximates the $2$-dimensional torus $\mathbb{T}^2$ as $n\to \infty$. Then the 
		spectral zeta function of $\widetilde{\Delta}_n$ admits an asymptotic expansion 
		as $n\to \infty$ for $s\in\C$ with $\Rel(s)\in(0,1)$ 
		\begin{equation}\begin{split}
		\zeta(\widetilde{\Delta}_n,s) \sim A(s)n^{\alpha-2s}+B(s) + \sum_{j=1}^\infty C_j(s)n^{-2j}, 
		\end{split}\end{equation}
		where the coefficients are explicitly computable and $B(s)$ is the spectral zeta function of the 
		Laplace Beltrami operator on $\mathbb{T}^2$. Moreover, if we define 
\begin{equation*}
H_n(s) := \pi^{-s}\Gamma(s)\Big(\zeta(\widetilde{\Delta}_n,s)-A(s)n^{\alpha-2s}\Big),
\end{equation*}
then the Epstein-Riemann conjecture in $2$ dimensions is equivalent\footnote{This 
equivalence fails if we instead of $\widetilde{\Delta}_n$ we consider the classical 
combinatorial Laplacian.} to 
\begin{equation}
\lim_{n\rightarrow\infty}\abs{\frac{H_n(1-s)}{H_n(s)}}=1.
\end{equation}
\end{Thm}

\noindent  In the remainder of this section we introduce 
the setting and formulate our main results explicitly. \medskip

\noindent \emph{Acknowledgements} The authors thank Anders Karlsson for crucial 
discussions of his work \cite{FrKa}, that inspired this paper. They also appreciate
the motivation during the early stages of this project.

\subsection{Combinatorial Laplacians on finite torus graphs}\label{discrete-lo}

For each integer $n\in\N$ we consider the modular quotient ring $\S^1_n:=\Z/n\Z$. 
As the notation suggests, this is viewed as a discretization of the smooth submanifold  $\S^1 \subset\R^2$.  

\begin{Def} Consider any function $u:\S^1_n\rightarrow\R$. 
\begin{enumerate}
\item The Dirac operator $D_n$ on $\S^1_n$ is defined by 
\begin{equation*}
	D_nu(k):=\frac{n}{2\pi}\left(u(k)-u(k-1)\right). 
\end{equation*}
\item The combinatorial Laplacian $\mathcal L_n$ on $\S^1_n$ is defined by 
\begin{equation*}
	\mathcal L_nu(k):=(D^t_nD_n)u(k)=\frac{n^2}{4\pi^2} \big( 2u(k)-u(k+1)-u(k-1)\big),
\end{equation*}
where the adjoint is defined with respect to the usual scalar product on functions
$u:\S^1_n\rightarrow\R$ identified with elements of $\R^n$.
\end{enumerate}
\end{Def}

Let $\alpha \in \N$ be a positive integer. 
The $\alpha$-dimensional discrete torus $\T^\alpha_n$ is defined as the Cartesian product of $\alpha$ discrete circles
$\bigtimes_{i=1}^\alpha \S^1_n$. The corresponding combinatorial Laplacian on $\T^\alpha_n$,
acting on functions $u:\T^\alpha_n\rightarrow \R$ is defined by
\begin{equation}\label{e-comp-lap}
	\Delta_n u(k):=\sum_{j=1}^\alpha \mathcal L_{n,j} u(k), \quad 
	\mathcal L_{n,j}:=\id\otimes\dots\otimes\id\otimes \mathcal L_n\otimes \id\otimes\dots\otimes\id,
\end{equation}
where $k=(k_1,\dots,k_\alpha)$ and $\mathcal L_n$ appears at the $j$-the place of the tensor product, i.e.
$\mathcal L_{n,j}$ is the combinatorial Laplacian on $\S^1_n$
acting on $k_j$. In the case $\alpha=2$ the combinatorial Laplacian is explicitly given by 
	\begin{equation}\begin{split}\label{e-5-point}
		\Delta_nu(k)&=\left(\mathcal L_n\otimes\id+\id\otimes\mathcal L_n\right)u(k)\\
		&=\frac{n^2}{4\pi^2}\big(4u(k_1,k_2)-u(k_1+1,k_2)-u(k_1-1,k_2)\\
		&-u(k_1,k_2+1)-u(k_1,k_2-1)\big)
	\end{split}\end{equation}
This corresponds up to the normalization of $\frac{n^2}{4\pi^2}$ to the $5$-point star Laplacian, known 
from finite difference methods in numerical analysis. Despite the presence of the normalizing factor, 
we will still refer to \eqref{e-5-point} as the $5$-point star Laplacian. In Figure \ref{fig:5-point} one sees a schematic 
representation of the $5$-point star Laplacian. Here the number on each node of the figure denotes 
weighting of this node, as in \eqref{e-5-point}.  
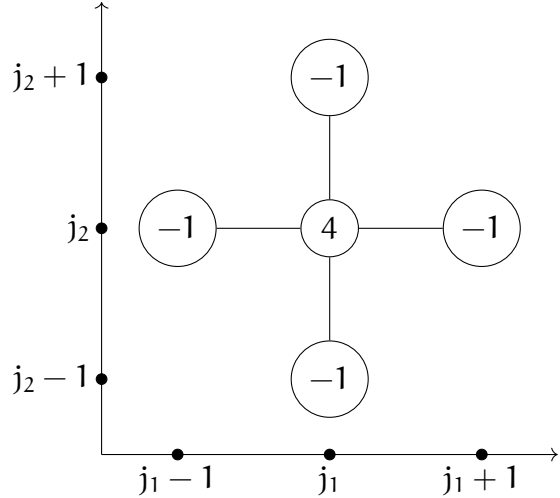
\begin{figure}[!htbp]
  \centering
	  \begin{tikzpicture}
	\node (A) at (7,0) [circle,draw] {$4$};
	\node (B) at (9,0) [circle,draw] {$-1$};
	\node (C) at (5,0) [circle,draw] {$-1$};
	\node (D) at (7,2) [circle,draw] {$-1$};
	\node (E) at (7,-2) [circle,draw] {$-1$};
    \draw[-] (A) to (D);
    \draw[-] (A) to (C);
    \draw[-] (A) to (B);
    \draw[-] (A) to (E);
        \draw[->](4,-3) to (10,-3);
     \draw[->](4,-3) to (4,3);
     \filldraw[black] (5,-3) circle (2pt)node[anchor=north] {$j_1-1$};
      \filldraw[black] (7,-3) circle (2pt)node[anchor=north] {$j_1$};
       \filldraw[black] (9,-3) circle (2pt)node[anchor=north] {$j_1+1$};
        \filldraw[black] (4,-2) circle (2pt)node[anchor=east] {$j_2-1$};
        \filldraw[black] (4,-0) circle (2pt)node[anchor=east] {$j_2$};
        \filldraw[black] (4,2) circle (2pt)node[anchor=east] {$j_2+1$};
  \end{tikzpicture}
    \caption{Schematic representation of the 5-point star Laplacian }
  \label{fig:5-point}
  \end{figure}

In our considerations, we also work with a slightly different operator, 
that can be viewed as a refinement of $\Delta_n$. The specific reason for its 
use are much better properties of the coefficients in the asymptotics of the associated spectral zeta function 
as $n \to \infty$. 

\begin{Def} The $9$-point star Laplacian acting on $u:\T^2_n\rightarrow \R$ is defined by
\begin{align}\label{e-lap-cor-ten}
	\widetilde\Delta_nu(k) :=\big(\mathcal L_n\otimes \id\big)u(k)+\big(\id\otimes\mathcal L_n\big)u(k)-\frac{2}{3}\frac{\pi^2}{n^2}\big(\mathcal L_n\otimes\mathcal L_n)u(k)
\end{align}
\end{Def}
Note that the first two summands are the usual combinatorial Laplacian, see \eqref{e-5-point} in two dimensions.  
A straightforward computation yields the following explicit form
\begin{equation}\begin{split}\label{e-9-point}
	\widetilde\Delta_n u(j_1,j_2):=&\frac{n^2}{4 \pi^2}\bigg( \frac{10}{3}u(j_1,j_2)-\frac{2}{3}\Big(u(j_1+1,j_2)+u(j_1,j_2+1)\\
	&+u(j_1-1,j_2)+u(j_1,j_2-1)\Big)-\frac{1}{6}\Big(u(j_1-1,j_2-1)\\
	&+u(j_1-1,j_2+1)+u(j_1+1,j_2-1)+u(j_1+1,j_2+1)\Big)\bigg).
\end{split}\end{equation}
Up to the $\frac{ n^2}{4\pi^2}$ factor, the operator corresponds to the compact 9-point star Laplacian, known from numerical analysis. 
This operator can be visualized as in Figure \ref{fig:f-5-9-point}. Here, exactly as in Figure \ref{fig:5-point}, the number on each node
denotes the weighting of this node, see \eqref{e-9-point}.

\begin{figure}[!htbp]

  \centering
	  \begin{tikzpicture}

	\node (A) at (7,0) [circle,draw] {$\frac {10} {3} $};
	\node (B) at (9,0) [circle,draw] {$-\frac{2}{3}$};
	\node (C) at (5,0) [circle,draw] {$-\frac{2}{3}$};
	\node (D) at (7,2) [circle,draw] {$-\frac{2}{3}$};
	\node (E) at (7,-2) [circle,draw] {$-\frac{2}{3}$};
	\node (F) at (9,-2) [circle,draw] {$-\frac{1}{6}$};
	\node (G) at (5,-2) [circle,draw] {$-\frac{1}{6}$};
	\node (H) at (9,2) [circle,draw] {$-\frac{1}{6}$};
	\node (I) at (5,2) [circle,draw] {$-\frac{1}{6}$};
    \draw[-] (A) to (D);
    \draw[-] (A) to (C);
    \draw[-] (A) to (B);
    \draw[-] (A) to (E);
     \draw[-] (A) to (F);
      \draw[-] (A) to (F);
       \draw[-] (A) to (G);
        \draw[-] (A) to (H);
         \draw[-] (A) to (I);
    \draw[->](4,-3) to (10,-3);
     \draw[->](4,-3) to (4,3);
     \filldraw[black] (5,-3) circle (2pt)node[anchor=north] {$j_1-1$};
      \filldraw[black] (7,-3) circle (2pt)node[anchor=north] {$j_1$};
       \filldraw[black] (9,-3) circle (2pt)node[anchor=north] {$j_1+1$};
        \filldraw[black] (4,-2) circle (2pt)node[anchor=east] {$j_2-1$};
        \filldraw[black] (4,-0) circle (2pt)node[anchor=east] {$j_2$};
        \filldraw[black] (4,2) circle (2pt)node[anchor=east] {$j_2+1$};
  \end{tikzpicture}\\

\caption{Schematic representation of compact 9-point star Laplacian }
\label{fig:f-5-9-point}
\end{figure}
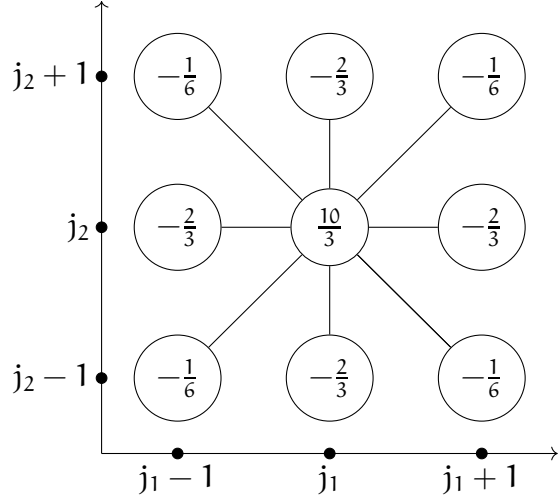

\subsection{Spectrum of combinatorial Laplacians on finite torus graphs}

The spectrum of the operator $\mathcal L_n$, represented by an $n \times n$ real-valued matrix, can be computed explicitly and is given by
(eigenvalues appear multiple times according to their multiplicity)
\begin{equation}\label{spectrum-L}
	\sigma(\mathcal L_n)=\bigg\{\left. \frac{n^2}{\pi^2} \sin^2\Big(\frac{\pi k}{n}\Big) \right|  k=0\, , \cdots, n-1\bigg\}.
\end{equation}
Note that the eigenvalue $\frac{n^2}{\pi^2} \sin^2\Big(\frac{\pi k}{n}\Big)$ with $k\neq 0$ has multiplicity two, i.e. 
appears twice in the enumeration above: once for $k$ and once again for $n-k$. The eigenvalue $0$ has multiplicity
one, i.e. it appears only once in the enumeration above: only for $k=0$. \medskip

\noindent The spectrum of $\Delta_n$ is then given in view of \eqref{spectrum-L} and \eqref{e-comp-lap} by
\begin{equation}
	\sigma(\Delta_n)=\bigg\{\left. \frac{n^2}{\pi^2}\sum_{j=1}^\alpha\sin^2\Big(\frac{\pi k_j}{n}\Big)\right| \forall_{j=1,\ldots, \alpha}: k_j = 0\, , \cdots, n-1\bigg\}
\end{equation}
where each eigenvalue appears multiple times according to its multiplicity. 
In the case $\alpha = 2$, we can similarly compute the spectrum of the $9$-point star Laplacian
from \eqref{spectrum-L} and \eqref{e-lap-cor-ten}.
This proves the formulae in the next proposition.

\begin{Prop}\label{spectrum-deltas}
\begin{enumerate}
\item The spectrum of the $5$-point star Laplacian $\Delta_n$ is given by
\begin{align*}
	\sigma(\widetilde\Delta_n)=\Bigg\{ \left. \frac{n^2}{\pi^2}\sin^2\Big(\frac{\pi k_1}{n}\Big)+\frac{n^2}{\pi^2}\sin^2\Big(\frac{\pi k_2}{n}\Big)
	\right| \forall_{j=1,2}: k_j = 0\, , \cdots, n-1\Bigg\}.
\end{align*}
\item The spectrum of the $9$-point star Laplacian $\Delta_n$ is given by
\begin{align*}
	\sigma(\widetilde\Delta_n)=\Bigg\{ &\frac{n^2}{\pi^2}\sin^2\Big(\frac{\pi k_1}{n}\Big)+\frac{n^2}{\pi^2}\sin^2\Big(\frac{\pi k_2}{n}\Big)\\
	&\left.-\frac{2n^2}{3\pi^2} \sin^2\Big(\frac {\pi k_1}{n}\Big)\sin^2\Big(\frac{\pi k_2}{n}\Big)\right| \forall_{j=1,2}: k_j = 0\, , \cdots, n-1\Bigg\}.
\end{align*}
\end{enumerate}
In both expressions, eigenvalues appear multiple times according to their multiplicity. 
\end{Prop}

\begin{Rem}
Note that despite a negative summand in Proposition \ref{spectrum-deltas} (2), 
the $9$-point star Laplacian $\Delta_n$ is still a non-negative operator. 
Indeed, its eigenvalues are of the form $\frac{n^2}{\pi^2} (a^2+b^2-\frac{2}{3}a^2b^2)$
with $|a|,|b| < 1$. These are non-negative, since 
\begin{align}\label{ab-play}
a^2+b^2-\frac{2}{3}a^2b^2 \geq a^2+b^2 - \frac{2}{3}a^2 = \frac{1}{3}a^2 + b^2 \geq 0.
\end{align}
\end{Rem}

\subsection{Spectral zeta functions of combinatorial Laplacians}

The spectral zeta function $\zeta(s, T)$ of a non-negative symmetric linear operator $T$ acting on a finite dimensional
real vector space is defined by summing over all its non-zero eigenvalues taken to the
$s$-th power, with $s \in \C$. Hereby, non-negativity of eigenvalues allows to use the 
principal branch of logarithm in the definition of the complex $s$ power. For 
$\Delta_n$ and $\widetilde\Delta_n$ we find in view of Proposition \ref{spectrum-deltas}
\begin{equation} \begin{split}
	\zeta (s,\Delta_n) &:= \sum_{k_1=0}^{n-1}\sum_{k_2=0}^{n-1} \bigg(
	 \frac{n^2}{\pi^2}\sin^2\Big(\frac{\pi k_1}{n}\Big)+\frac{n^2}{\pi^2}\sin^2\Big(\frac{\pi k_2}{n}\Big)
	\bigg)^{-s}, \\
 \zeta (s,\widetilde\Delta_n) &:= \sum_{k_1=0}^{n-1}\sum_{k_2=0}^{n-1} \biggl( \frac{n^2}{\pi^2}\sin^2\Big(\frac{\pi k_1}{n}\Big)+\frac{n^2}{\pi^2}\sin^2\Big(\frac{\pi k_2}{n}\Big)\\
	& \qquad \qquad  \quad -\frac{2n^2}{3\pi^2} \sin^2\Big(\frac {\pi k_1}{n}\Big)\sin^2\Big(\frac{\pi k_2}{n}\Big)\bigg)^{-s},
\end{split}\end{equation}
where we exclude $(k_1,k_2) = (0,0)$ from summation. 

\subsection{Hadamard regularized limits and integrals}

In order to define spectral zeta function of the Laplace Beltrami operator, that is the 
analytic counterpart of the combinatorial $5$- and $9$-point star Laplacians in the limit as $n \to \infty$, 
we need to introduce Hadamard regularized integrals. \medskip

\noindent Consider function $u\in C^\infty(0,\infty)$ such that for $x\rightarrow\infty$ 	
  \begin{align}\label{e-class-asymp}
		u(x)=\sum_{j=1}^N\sum_{k=0}^{n_j} a_{jk}x^{\alpha_j}\log^k(x)+\sum_{k=0}^{n_0} a_{0k}\log^k(x)+o(x^{\alpha_N}\log^{n_N}(x))
\end{align}
for some $N\in\N$, with $(\alpha_j)_{j=1,\dots,N}\subset\C$ with $\Rel(\alpha_j)$ monotonously decreasing and $\Rel(\alpha_N)<0$.
We define the regularized limit of $u$ as $x\rightarrow\infty$ by the constant term in the asymptotic
	\begin{equation*}
		\underset{x\rightarrow\infty}{\mathrm{LIM}} \ u(x) := a_{00}.
	\end{equation*}
If $\Rel(\alpha_N)<-1$, the integral of $u$ over $[1,R]$ has an asymptotic expansion 
of the same form as in \eqref{e-class-asymp} for $R \to \infty$. In that case we define the regularized integral as follows
	 \begin{align*}
		\rint_1^\infty u(x)dx := \underset{R\rightarrow\infty}{\mathrm{LIM}}\int_1^R u(x)dx
	\end{align*} 
Similarly, if $u(x)$ as an asymptotic expansion for $x \to 0$ of the form
 \begin{align}\label{e-class-asymp0}
		u(x)=\sum_{j=1}^{N'}\sum_{k=0}^{n_j} b_{jk}x^{\alpha_j}\log^k(x)+\sum_{k=0}^{n_0} b_{0k}\log^k(x)+o(x^{\alpha_{N'}}\log^{n_{N'}}(x))
\end{align}
for some $N'\in\N$, with $(\alpha_j)_{j=1,\dots,N'}\subset\C$ with $\Rel(\alpha_j)$ monotonously increasing and $\Rel(\alpha_{N'})>0$,
then we can define the regularized limit of $u$ at zero. 
In the same way, if $\Rel(\alpha_{N'})>-1$, we also can define the regularized integral on $[0,1]$. If both regularized integrals are well-defined,
we may write
\begin{equation}
	\rint_0^\infty u(x)dx:= \rint_0^1 u(x)dx+\rint_1^\infty u(x)dx.
\end{equation}
The regularized integral satisfies a peculiar change of variables rule.
\begin{Prop}[\cite{Les} Lemma 2.1.4] \label{reg-int-change}
Assume $u$ satisfies \eqref{e-class-asymp} and \eqref{e-class-asymp0} with $\Rel(\alpha_N)<-1$
and $\Rel(\alpha_{N'})>-1$. Then for $\lambda>0$ the following holds
\begin{equation*}
\rint_0^\infty u(\lambda x)dx=\lambda^{-1}
\bigg(\rint_0^\infty u(x)dx +\sum_{k=0}^{m_j}b_{jk}\frac{\log^{k+1}(\lambda)}{k+1}-
\sum_{k=0}^{n_{j'}}a_{j'k}\frac{\log^{k+1}(\lambda)}{k+1}\bigg)
\end{equation*}
where $a_{j'k}$ is the coefficient to the term $x^{-1}\log^k(x)$ in the asymptotic expansion of $u$ as 
$x\rightarrow \infty$ and $b_{jl}$ is the coefficient to the term $x^{-1}\log^k(x)$ in the asymptotic 
expansion of $u$ as $x\rightarrow 0$. 
\end{Prop}

\begin{Rem}
If in both asymptotic expansions \eqref{e-class-asymp} and \eqref{e-class-asymp0}
no $x^{-1}\log^k x$ terms exist, then it is just the usual change of variables rule. 
\end{Rem}

\subsection{Spectral zeta function of the Laplace Beltrami operator on a torus}

Now let us consider the Laplace Beltrami operator $\Delta$ on the 
$\alpha$-dimensional torus $\T^\alpha=\bigtimes_{i=1}^\alpha \S^1$. Similar to
\eqref{e-comp-lap}, we can write $\Delta$ in terms of the Laplace-Beltrami 
operator $\Delta_{\mathbb{S}^1}$, acting on the individual components $\S^1$
$$
\Delta \, u(x_1\, ,\cdots, x_\alpha) = \sum_{j=1}^\alpha \Delta_{\mathbb{S}^1, j} u(x_1\, ,\cdots, x_\alpha),
$$
where $\Delta_{\mathbb{S}^1, j}$ is the operator $\Delta_{\mathbb{S}^1}$, acting on $x_j$. 
The functions $\left(\frac{1}{\sqrt{2\pi}}e^{ik x}\right)_{k \in \Z}$
form an orthonormal basis of $L^2(\mathbb{S}^1)$ of eigenfunctions 
of $\Delta_{\mathbb{S}^1}$ to the eigenvalues $k^2, k \in \Z$. 
The eigenvalues $k^2\neq 0$ have multiplicity two, the eigenvalue $k^2=0$ has multiplicity one. 
Thus the spectrum of $\Delta$ is given by
 \begin{equation*}
 	\sigma(\Delta)=\bigg\{ k_1^2+\dots+k_\alpha^2: k=(k_1,\dots,k_\alpha)\in\Z^\alpha \bigg\},
 \end{equation*}
 where the eigenvalues apear multiple times according to their multiplicity.\medskip
 
The spectral zeta function $\zeta(\Delta, s)$ of $\Delta$ is defined by summing over non-zero 
eigenvalues, according to their multiplicity, with $(-s)$-exponent. That sum converges
and is holomorphic for $\Rel(s)>\frac \alpha 2$. In this work we will use the following integral 
representation for $\Rel(s)>\frac \alpha 2$, cf. \cite[(1.24), (1.26)]{LesVer}

\begin{equation}\label{e-zeta-reg-int} \begin{split}
\zeta(\Delta, s) &= \sum_{k\in\Z^\alpha\setminus\{0\}} (k_1^2 + \cdots + k_\alpha^2)^{-s} \\
&= 2 \, \frac{\sin(\pi s)}{\pi}\frac{\Gamma(1-s)\Gamma(\alpha)}{\Gamma(\alpha-s)}
\rint_0^\infty z^{2\alpha-2s-1}\Tr(\Delta+z^2)^{-\alpha}dz,
\end{split}\end{equation}
where in the first sum we used the multi-index notation $k=(k_1,...,k_\alpha)$.
In the second expression, $(\Delta+z^2)^{-\alpha}$ denotes $\alpha$-th power of the resolvent of 
$\Delta$, $\Tr$ the trace, and $\Gamma$ the Gamma function. The front factor will be abbreviated by 
\begin{equation}\label{VA}
V_\alpha(s)=2 \, \frac{\sin(\pi s)}{\pi} \frac{\Gamma(1-s)\Gamma(\alpha)}{\Gamma(\alpha-s)}.
\end{equation}
The standard asymptotic expansion of the resolvent  trace $\Tr(\Delta+z^2)^{-\alpha}$ 
gives a meromorphic continuation of \eqref{e-zeta-reg-int} to the whole complex plane.
Classical references on the spectral zeta function of $\Delta$ and its applications are e.g.
\cite{MiPl, Se, Haw}.
\medskip

Exactly the same integral representation \eqref{e-zeta-reg-int} holds for the 
spectral zeta functions $\zeta (s,\Delta_n)$ and $ \zeta (s,\widetilde\Delta_n)$
of the combinatorial $5$- and $9$-point star Laplacians, where $\alpha = 2$ and
$\Tr(\Delta+z^2)^{-\alpha}$ is replaced by 
\begin{equation}\label{resolvent-traces-discrete}\begin{split}
        \Tr(\Delta_n+z^2)^{-2} &= \sum_{k_1=0}^{n-1}\sum_{k_2=0}^{n-1}
        \bigg(\frac{n^2}{\pi^2}\sin^2\Big(\frac{\pi k_1}{n}\Big)+\frac{n^2}{\pi^2}\sin^2\Big(\frac{\pi k_2}{n}\Big)+z^2\bigg)^{-2}, \\
	\Tr(\widetilde\Delta_n+z^2)^{-2} &= \sum_{k_1=0}^{n-1}\sum_{k_2=0}^{n-1}\bigg(\frac{n^2}{\pi^2}\sin^2\Big(\frac{\pi k_1}{n}\Big)
	+\frac{n^2}{\pi^2}\sin^2\Big(\frac{\pi k_2}{n}\Big)\\
	&-\frac{2n^2}{3\pi^2} \sin^2\Big(\frac {\pi k_1}{n}\Big)\sin^2\Big(\frac{\pi k_2}{n}\Big)+z^2\bigg)^{-2},
\end{split}\end{equation}
respectively. The explicit structure of these combinatorial resolvent traces follows directly from 
the explicit form of the eigenvalues in Proposition \ref{spectrum-deltas}.

\subsection{Riemann- and Epstein zeta function}\label{epstein-zeta}

The spectral zeta function on the torus $\T^\alpha$ is a special case of a 
wide class of zeta functions, namely the Epstein zeta functions. Those functions 
are defined for any real-valued positive definite $\alpha \times \alpha$-matrix $Q$ by
\begin{equation}
	\zeta_Q(s) :=\sum_{k\in\Z^\alpha\setminus\{0\}}(k^TQk)^{-s}, \quad \textup{Re} (s) > \frac{\alpha}{2}.
\end{equation}
If $Q = \id$ is the identity matrix, we recover the spectral zeta function of the torus
\begin{equation}\label{epstein-torus}
	\zeta_\id(s)=\zeta (\Delta, s),
\end{equation}
for all $\alpha\in\N$. Epstein zeta functions generalize many concepts of the Riemann zeta function. 
First of all, $\zeta_Q(s)$ can be continued to a meromorphic function to the whole complex plane. 
The poles are simple poles at $s=0$ and $s=\alpha/2$. The Epstein zeta function also has trivial zeros, located at $s=-z$ for $z\in\N$. \medskip

A very common technical tool in working with the Epstein zeta function is the \emph{complete Epstein zeta function}, also called the Epstein xi function.
It is defined by
\begin{equation}\label{e-ezeta-comp}
	\xi_Q(s) := \pi^{-s}\Gamma(s)\zeta_Q(s).
\end{equation}
The complete Epstein zeta function satisfies the functional equation 
\begin{equation}\label{e-ezeta-func}
	\xi_Q(s)= (\det Q)^{-1/2}\xi_{Q^{-1}} (\alpha/2-s),
\end{equation}
which is due to \cite{Eps}. This relation indicates some interesting behaviour of $\zeta_Q(s)$ 
in the so-called \emph{critical strip} $\{s\in\C: 0<\Rel(s)<\alpha/2\}$. Zeros located in the critical strip are called 
\emph{critical zeros}. Similar to the Riemann conjecture, we can formulate a conjecture 
for the Epstein zeta function\medskip

\begin{Hyp}[Epstein Riemann (abbreviated as E.R.) conjecture]\label{h-eprie} \ \medskip

\noindent All critical zeros $s\in\C$ \textup{(}i.e. zeros with $0<\Rel(s)<\alpha/2$\textup{)} of the Epstein zeta function $\zeta_\id(s)$ have real part $\textup{Re}(s) = \alpha / 4$. 
\end{Hyp}

In dimensions $\alpha=2,4,6$ and $8$,
the Epstein zeta function $\zeta_\id(s)$ (and hence by \eqref{epstein-torus} also $\zeta (\Delta, s)$ $-$ the spectral zeta function
on a torus) can be expressed in terms of the Riemann zeta function $\zeta_R(s)$, see \cite{Zuc}. For example, in case $\alpha = 2$ we have the beautiful formula
due to Glasser, see \cite[(1.1)]{Zuc}
\begin{equation}\label{e-ep-rie-beta}
	\zeta_\id(s) \equiv \zeta (\Delta, s)=4\zeta_R(s)\beta(s),
\end{equation}
where $\beta(s)=\sum_{k=1}^\infty \frac{(-1)^{k-1}}{(2k-1)^s}$ the Dirichlet beta function, 
that can be viewed as a Dirichlet $L$-function of a Dirichlet-character $\chi$. 
It is conjectured by numerical evidence that the non-trivial zeros of $\beta(s)$ are located on the critical 
line\footnote{In fact \cite{Phe} claims this is equivalent to the Riemann. conjecture.} $\textup{Re}(s) = 1/2$, cf.
e.g. \cite{Kaw}. Provided this were true, for $\alpha = 2$, the Epstein-Riemann-conjecture and the Riemann conjectures were equivalent. 
\medskip

For $\alpha \geqslant 3$ the Epstein-Riemann (E.R.) conjecture fails, as is demonstrated by the formulae in 
\cite[(1.4) -- (1.6)]{TrSa}, as well as other numerical computations in \cite[Section 5]{TrSa}. We summarize this as the answer to the
Conjecture \eqref{h-eprie}.

\begin{Ans}\label{answer}
\begin{enumerate}
\item For $\alpha = 1$ the E.R. conjecture is the Riemann conjecture.
\item For $\alpha = 2$ the E.R. conjecture is equivalent to the Riemann conjecture, if
zeros of Dirichlet beta function $\beta(s)$ are located on the critical line $\textup{Re}(s) = 1/2$.
\item For $\alpha \geqslant 3$ the E.R. conjecture is wrong. However it may still hold 
for imaginary part of critical zeros being sufficiently large. 
\end{enumerate}
\end{Ans}

The main contribution of this work is the relation between the E.R. conjecture and the
asymptotics of $\zeta(\widetilde{\Delta}_n, s)$ in dimension $\alpha = 2$, with a clear path
for generalization to higher dimensions.

\subsection{Statement of the main results}\label{main-stat}
For an arbitrary $\alpha\in\N$, the eigenvalues of the combinatorial Laplacians $\Delta_n$
and $\widetilde{\Delta}_n$ converge to the eigenvalues of the Laplace Beltrami operator $\Delta$ as $n \to \infty$.
Hence for any $s\in \C$ with $\textup{Re}(s)> \alpha /2$
$$
\lim_{n\to \infty}\zeta(\Delta_n,s) = \lim_{n\to \infty} \zeta(\widetilde{\Delta}_n,s) =  \zeta(\Delta,s).
$$ 
However, for $s$ in the critical strip 
$0<\Rel(s)<\alpha/2$, convergence fails and is replaced by an intricate asymptotic 
expansion. For $\alpha = 1$, in \cite[Theorem 3]{FrKa} the authors have shown a partial 
asymptotics of the following form\footnote{The formula in \cite[Theorem 3]{FrKa} differs from the
presentation here by normalization.}

\begin{Thm}[\cite{FrKa} Theorem 3] \label{asymp-circ}
	For $\Rel(s)\in(0,1/2)$ we have
	\begin{align*}
		\zeta(\mathcal{L}_n, s) 
		&= \pi^{2s-1/2} \frac{\Gamma(1/2-s)}{\Gamma(1-s)} n^{1-2s}+2\zeta_R(2s)+\frac{2s}{3}\pi^2 \zeta_R(2s-2)n^{-2}+o(n^{-2}) \\
		&= \pi^{2s-1/2} \frac{\Gamma(1/2-s)}{\Gamma(1-s)} n^{1-2s}+\zeta(\Delta_{\mathbb{S}^1},s)+\frac{s}{3}\pi^2 
		\zeta (\Delta_{\mathbb{S}^1},s-1)n^{-2}+o(n^{-2}).
	\end{align*}
\end{Thm}

For an arbitrary $\alpha\in\N$ in \cite[Theorem 1]{FrKa} the authors established a one term shorter partial
asymptotic expansion, with the first two terms explicitly identified and existence \& structure of the next term left as an open question. 
Our first main result addresses this open question for $\alpha = 2$.
   
\begin{Thm}\label{t-asymp}
For $\Rel(s) \in (0, \alpha/2)$ we have for any integer $M \in \N$ as $n\to \infty$	
\begin{equation} \label{asymptotics-both-laplacians}\begin{split}
&\zeta(\Delta_n, s) = V_\alpha(s) \Bigg( a(s) \, n^{\alpha-2s}+\sum_{m=0}^{M-1} b_m(s) \, n^{-2m} \Bigg)
+ O(n^{-2M -2s +2}), \\
&\zeta(\widetilde{\Delta}_n, s) = V_\alpha(s) \Bigg( \widetilde{a}(s) \, n^{\alpha-2s}+\sum_{m=0}^{M-1} \widetilde{b}_m(s) \, n^{-2m} \Bigg)
+ O(n^{-2M -2s +2}).
\end{split}\end{equation}
The leading coefficients $a(s)$ and $\widetilde{a}(s)$ are explicitly given by 
\begin{equation}\begin{split}
a(s) & =  \int_0^\infty z^{2\alpha-2s-1}  \int_0^1 \int_0^1 
\Big(\frac{\sin^2(\pi x)}{\pi^2}+\frac{\sin^2(\pi y)}{\pi^2}+z^2\Big)^{-\alpha}  dx dy dz, \\
\widetilde{a}(s) &=  \int_0^\infty z^{2\alpha-2s-1}  \int_0^1 \int_0^1 
\Bigg( \frac{\sin^2(\pi x)}{\pi^2}+\frac{\sin^2(\pi y)}{\pi^2} + z^2 \\  &-\frac{2 n^2}{3\pi^2}\sin^2\Big(\frac{\pi x}{n}\Big)\sin^2\Big(\frac{\pi y}{n}\Big) \Bigg)^{-\alpha}  dx dy dz.
\end{split}\end{equation}
The first two higher order coefficients are explicitly given by 	
\begin{equation}\begin{split}
b_0(s) &= \widetilde{b}_0(s) =  V_\alpha(s)^{-1} \zeta(\Delta,s), \quad
\widetilde{b}_1(s) = \frac{s\pi^2}{3} V_\alpha(s)^{-1} \zeta(\Delta,s-1), \\
b_1(s) &= \frac{s\pi^2}{3} V_\alpha(s)^{-1} \zeta(\Delta,s-1) + \frac{4\pi^2 V_2(s)}{(\alpha-s)}
\rint_0^\infty z^{4-2s+1} \! \! \!  \sum_{k_1, k_2 \in \Z}
\frac{k_1^2k_2^2}{(k_1^2+k_2^2+z^2)^{2}}.
\end{split}\end{equation}
\end{Thm}

\begin{Rem}
The computations in Section \ref{asymp-disc-zeta} can be adapted in a completely analogous way to higher dimensions. 
For this reason, we will continue writing $\alpha$ instead of the explicit $2$, as above. 
Naturally, our techniques may be applied to the one-dimensional case as well and would reprove the 
result from \cite{FrKa}. 
\end{Rem}
	
Our second main result relates the asymptotics of Theorem \ref{t-asymp} to the Epstein-Riemann conjecture. 
For this it is crucial for the $n^{-2}$ coefficient in \eqref{asymptotics-both-laplacians} to satisfy some functional equation. The coefficient $b_1(s)$
contains an additional term, known as an angular lattice sum, cf. \cite{BGM}. Due to this term, $b_1(s)$
does not admit a functional relation. This is why we have introduced a "corrected" discrete
Laplacian\footnote{The authors were surprised to see that the necessary correction that annihilates the 
angular lattice term, in fact leads to the well-known $9$-point star Laplacian.} $\widetilde{\Delta}_n$ with 
coefficient $\widetilde{b}_1(s)$, which no longer has any angular lattice sum term and 
satisfies a functional relation. Therefore we use $\widetilde{\Delta}_n$ to obtain an equivalent reformulation of
the E.R. conjecture.

\begin{Thm}\label{ER-equiv} Consider \eqref{asymptotics-both-laplacians} and define
for $s\in\C$ with $\Rel(s)\in(0,1)$
\begin{equation*}
H_n(s) := \pi^{-s}\Gamma(s)\Big(\zeta(\widetilde{\Delta}_n,s)-V_\alpha(s) \widetilde a(s)n^{\alpha-2s}\Big)
\end{equation*}
Then, the Epstein-Riemann conjecture \ref{h-eprie} for $\alpha = 2$ is equivalent to 
\begin{equation}
\lim_{n\rightarrow\infty}\abs{\frac{H_n(1-s)}{H_n(s)}}=1.
\end{equation}
\end{Thm}

By the Answer \ref{answer} the Epstein-Riemann conjecture for $\alpha = 2$ is related to the
Riemann conjecture and the generalized Riemann conjecture for the. Dirichlet beta function. 
We hope that our arguments apply to derive similar reformulations for $\alpha \geqslant 3$.
We expect that in higher dimensions further refinements of the discrete Laplacian will be necessary. \medskip

Our result fits the work \cite{FrKa} and \cite{Fri}, where similar statements are proved for the Riemann 
zeta function and certain Dirichlet L-functions. The case for $\alpha=2$ shows the most similarities with 
the statements in these previous works. As discussed above,  other Epstein zeta functions do not 
satisfy an Epstein-Riemann conjecture. It might be interesting to consider discrete spectral zeta functions in 
higher dimensions to examine how the E.R. conjecture has to be adapted.
	
\section{Euler-Maclaurin formula and the discrete resolvent trace}\label{asymp-disc-zeta}

\subsection{Generalities on the Euler-Maclaurin-formula}
In this section, we derive integral representations of the discrete resolvent traces in 
\eqref{resolvent-traces-discrete}, using the classical Euler-Maclaurin-formula iteratively. 
Let us recall the latter here for convenience.
	\begin{Thm}\label{emf}
	For any $M,n\in\N$ and $u\in C^{2M+1}([0,n])$ we have the identity
	\begin{align}\label{eemf}
		\sum_{i=0}^{n} u(i)&=\int_0^n u(x) dx+\frac{1}{2}\big(u(0)+u(n)\big)+\sum_{j=1}^{M}\frac{B_{2j}}{(2j)!}\big(u^{(2j-1)}(n)-u^{(2j-1)}(0)\big)\notag\\&
			+\frac{1}{(2M+1)!} \int_0^n B_{2M+1}(x-[x]) u^{(2M+1)}(x)dx,
	\end{align}
	where the $B_{k}$'s are the k-th Bernoulli-numbers and $B_k(x)$ denotes the Bernoulli-polynomial of order $k$. 
	\end{Thm}

There are various generalizations of the Euler-Maclaurin-formula. We mention those that are most
closely related to the problem studied here. First, \cite{LyMH} derived an Euler-Maclaurin-formula 
for summation in multiple variables. Application of that formula would provide an alternative argument 
for our problem here. However, one encounters a problem to control the decay of the error-terms. 
Moreover, for an arbitrary number of summations, the formulas become cumbersome. \medskip
	 
Another generalization of the Euler-Maclaurin-formula was formulated in \cite{Sid} and \cite{MoLy}. 
In \cite{Sid} functions with algebraic singularities at the endpoints of the interval are allowed. 
In this generalization, the integral part of the expansion is allowed to exist in the Hadamard finite part sense. 
Those algebraic singularities are precisely those that appear here as well. However, application of \cite{Sid}
then becomes problematic in the second iteration. This is why we only use the classical statement as in Theorem \ref{emf} iteratively. 
\medskip

There are also combinations of these two generalizations e.g. \cite{VPV}, and many more. 
We do not attempt to compile a complete list. \medskip
	
Theorem \ref{emf} implies by subtracting $u(n)$ the following formula
	\begin{align}\label{eemfup}
		\sum_{i=0}^{n-1} u(i)&=\int_0^n u(x) dx+\frac{1}{2}\big(u(0)-u(n)\big)+\sum_{j=1}^{M}\frac{B_{2j}}{(2j)!}\big(u^{(2j-1)}(n)-u^{(2j-1)}(0)\big)\notag\\&
			+\frac{1}{(2M+1)!} \int_0^n B_{2M+1}(x-[x]) u^{(2M+1)}(x)dx.
	\end{align}
We introduce an auxilliary notation for the terms in \eqref{eemfup}
	\begin{equation}\label{IADE}\begin{split}
		&I^{n}_{x_i}u :=\int_0^n u(x_i)dx_i, \quad
		A^{n}_{x_i} := \frac{1}{2}\Big(u\big(x_i=n\big)-u\big(x_i=0\big)\Big), \\
		&D_{x_i,M}^{n}u :=\sum_{j=1}^{M}\frac{B_{2j}}{(2j)!}\Big(\partial_{x_i}^{2j+1}u\big(x_i=n\big)-\partial_{x_i}^{2j+1}u\big(x_i=0\big)\Big), \\
		&E^{n}_{x_i,M}u :=\frac{1}{(2M+1)!} \int_0^n B_{2M+1}(x_i-[x_i]) \, \partial_{x_i}^{2M+1}u(x_i) dx.
	\end{split}\end{equation}
So \eqref{eemfup}, with summation over $x_i$, can be written in the short form as 
	\begin{equation}\label{e-emf-op}
		\sum_{j=0}^{n-1} u(x_i=j)=I^{n}_{x_i}u+A^{n}_{x_i}u+D_{x_i,M}^{n}u+E^{n}_{x_i,M}u.
	\end{equation}
	
\subsection{Estimates of the resolvent trace of $5$-point star Laplacian $\Delta_n$}\label{emf-res} \ \medskip

\noindent In view of \eqref{resolvent-traces-discrete}, we simplify notation below by introducing 
\begin{equation}\label{f-defn}\begin{split}
f(x,y,n,z) := \frac{n^2}{\pi^2}\sin^2\Big(\frac{\pi x}{n}\Big)+\frac{n^2}{\pi^2}\sin^2\Big(\frac{\pi y}{n}\Big)+z^2.
\end{split}\end{equation}
We want to apply \eqref{eemfup} to the resolvent trace in \eqref{resolvent-traces-discrete} for 
the classical ($5$-point star) combinatorial Laplacian $\Delta_n$ on $\T^2_n$. 
\begin{align*}
	\Tr(\Delta_n+z^2)^{-\alpha} &=\sum_{k_1 =0}^{n-1} \ \sum_{k_2 =0}^{n-1} f(k_1,k_2,n,z)^{-\alpha}.
\end{align*}
	
 \begin{Prop}\label{trace-asymp}
 	For any $z,\alpha>0$ and positive integers $M$ and $n$, the resolvent trace $\Tr(\Delta_n+z^2)^{-\alpha}$ 
	has the representation 
 	\begin{equation}\label{e-trace-asymp}\begin{split}
 			\textup{Tr}(\Delta_n+z^2)^{-\alpha} &= n^{-2\alpha+2} \int_0^1 \int_0^1 f(x, y,1,z/n)^{-\alpha} dx dy \\
			&+ 2\, I^{n}_{y}\circ E^{n}_{x,M} f^{-\alpha} + E^{n}_{y,M}\circ E^{n}_{x,M}f^{-\alpha},
 	\end{split}\end{equation}
where for some constant $C>0$, depending only on $M$
\begin{equation}\label{error5}\begin{split}
\Big| \, I^{n}_{y}\circ E^{n}_{x,M} f^{-\alpha} \Big| 
&\leqslant  C \, n^{-2M-2\alpha+1}\,  \Big(\frac{z}{n}\Big)^{-2\alpha-1} \left(\sqrt{ \Big(\frac{z}{n}\Big)^2+\frac{1}{\pi^2}} \ \right)^{-1}, \\
\Big| \, E^{n}_{y,M}\circ E^{n}_{x,M}f^{-\alpha} \Big|
&\leqslant  C \, n^{-4M-2\alpha}\, \Big(\frac{z}{n}\Big)^{-2\alpha-2}\left(\sqrt{ \Big(\frac{z}{n}\Big)^2+\frac{1}{\pi^2}} \ \right)^{-2}.
\end{split}\end{equation} 
 \end{Prop}
 
 \begin{proof}
We compute the individual operators in \eqref{IADE} applied to $f(x,y,n,z)^{-\alpha}$. 
Since $f(0,y,n,z)=f(n,y,n,z)$, we find
$$A^{n}_xf(x,y,n,z)^{-\alpha}=0. $$
For every odd derivative of $f(x,y,n,z)^{-\alpha}$ we get a factor $\sin(\frac{\pi x}{n})$ in all summands. These vanish at $x=0$ and $x=n$. Thus
$$D^{n}_{x,M}f(x,y,n,z)^{-\alpha}=0.$$
Applying now \eqref{eemfup}, we conclude
\begin{align*}
\Tr(\Delta_n+z^2)^{-\alpha}  =  \sum_{k_2=0}^{n-1}I^{n}_{x}f(x,k_2,n,z)^{-\alpha} 
+  \sum_{k_2=0}^{n-1}E^{n}_{x,M}f(x,k_2,n,z)^{-\alpha}.
\end{align*}
Repeat the same argument, applying \eqref{e-emf-op} to both remaining sums. 
As before we have that $D^{n}_{y,M}f(x,y,n,z)^{-\alpha}=0$ and $A^{n}_yf(x,y,n,z)^{-\alpha}=0$. Hence we arrive at

\begin{equation}\label{e-trace-near}\begin{split}
		\Tr(\Delta_n+z^2)^{-\alpha}  = I^{n}_{y}\circ I^{n}_{x}f^{-\alpha} +2\cdot I^{n}_{y}\circ E^{n}_{x,M}f^{-\alpha}
		+E^{n}_{y,M}\circ E^{n}_{x,M}f^{-\alpha} 	
\end{split}\end{equation} 
By a change of variables $\widetilde{x} = x/n$ and $\widetilde{y} = y/n$, the 
individual summands in \eqref{e-trace-near} can be written as follows
(note that $\partial_{\widetilde{x}} = n^{-1} \partial_{x}$ and $\partial_{\widetilde{y}} = n^{-1} \partial_{y}$)

\begin{align*}
I^{n}_{y}\circ I^{n}_{x}f^{-\alpha} &= n^{-2\alpha+2} \int_0^1 \int_0^1 f(\widetilde{x}, \widetilde{y},1,z/n)^{-\alpha} d\widetilde{x} d\widetilde{y}, \\
I^{n}_{y}\circ E^{n}_{x,M} f^{-\alpha} &= \frac{n^{-(2M+1)-2\alpha+2}}{(2M+1)!} \int_0^1 \int_0^1 B_{2M+1}(n\widetilde{x}-[n\widetilde{x}]) 
\partial_{\widetilde{x}}^{2M+1}f(\widetilde{x}, \widetilde{y},1,z/n)^{-\alpha} d\widetilde{x} d\widetilde{y}, \\
E^{n}_{y,M}\circ E^{n}_{x,M}f^{-\alpha} &= \frac{n^{-2(2M+1)-2\alpha+2}}{(2M+1)!} \int_0^1 \int_0^1 B_{2M+1}(n\widetilde{x}-[n\widetilde{x}]) 
B_{2M+1}(n\widetilde{y}-[n\widetilde{y}]) \\
&\qquad \qquad \qquad \qquad \qquad \quad \partial_{\widetilde{x}}^{2M+1} \partial_{\widetilde{y}}^{2M+1} 
f(\widetilde{x}, \widetilde{y},1,z/n)^{-\alpha} d\widetilde{x} d\widetilde{y}.
\end{align*} 
It remains to estimate the last two terms. We write $x,y$ for 
$\widetilde{x}, \widetilde{y}$, respectively and note for the partial derivatives of $f(x,y,1,z/n)^{-\alpha}$
	\begin{equation}\label{derivatives-f}\begin{split}
		\partial_x^{2M+1} f(x,y,1, z/n)^{-\alpha} &= \sum_{i=0}^{2M} A_i(x,y)f(x,y,1,z/n)^{-\alpha-i-1}, \\
		\partial_x^{2M+1}\partial_y^{2M+1} f(x,y,1, z/n)^{-\alpha} &= \sum_{i=0}^{4M} B_i(x,y)f(x,y,1,z/n)^{-\alpha-i-2},
	\end{split}\end{equation}
where $A_i, B_i\in C^\infty ([0,1]^2)$ are independent of $z$ and $n$. 
We estimate $A_i, B_i$ and the Bernoulli polynomials $B_{2M+1}$ against a constant, 
$f(x,y,1,z/n)^{-\alpha-i}$ against $(z/n)^{-2\alpha}$, and obtain
for some constant $C>0$, depending only on $M$
\begin{equation}\label{IE-EE-estimates}\begin{split}
\Big| \, I^{n}_{y}\circ E^{n}_{x,M} f^{-\alpha} \Big| 
&\leqslant  C \, \frac{n^{-2M-2\alpha+1}}{(2M+1)!} \Big(\frac{z}{n}\Big)^{-2\alpha}\int_0^1 \Big(\frac{\sin^2(\pi x)}{\pi^2}+\frac{z^2}{n^2}\Big)^{-1}dx,\\
\Big| \, E^{n}_{y,M}\circ E^{n}_{x,M}f^{-\alpha} \Big|
&\leqslant  C \, \frac{n^{-4M-2\alpha}}{(2M+1)!} \Big(\frac{z}{n}\Big)^{-2\alpha} \int_0^1 \Big(\frac{\sin^2(\pi x)}{\pi^2}+\frac{z^2}{n^2}\Big)^{-1}dx
\\ & \qquad \qquad \qquad \qquad \quad \times \int_0^1 \Big(\frac{\sin^2(\pi y)}{\pi^2}+\frac{z^2}{n^2}\Big)^{-1}dy.
\end{split}\end{equation}
By \cite{ZMGR} pp. 177, 2.562 we know that for $b/a>-1$ 
\begin{equation}\label{trigonometry}
		\int \frac{dx}{a+b\sin^2(x)}=\frac{\mathrm{sign} (a)}{\sqrt{a(a+b)}}\arctan\Big(\sqrt{\frac{a+b}{a}}\tan(x)\Big).
\end{equation}
Since $\tan(\pm\frac{\pi}{2})=\pm\infty$ and $\arctan(\pm \infty)= \pm \frac{\pi}{2}$ we get using
$\sin^2(\pi x)\geqslant 0$ for all $x\in[0,1]$ and symmetry
\begin{align}\label{explicit-integral}
	\int_0^1 \Big(\frac{\sin^2(\pi x)}{\pi^2}+\frac{z^2}{n^2}\Big)^{-1}dx = \frac{1}{\frac{z}{n}\sqrt{(\frac{z}{n})^2+\frac{1}{\pi^2}}}.
\end{align}
Summarizing, we have shown for some constant $C>0$, depending only on $M$
\begin{align*}
I^{n}_{y}\circ I^{n}_{x}f^{-\alpha} &= n^{-2\alpha+2} \int_0^1 \int_0^1 f(x, y,1,z/n)^{-\alpha} dx dy, \\
\Big| \, I^{n}_{y}\circ E^{n}_{x,M} f^{-\alpha} \Big| 
&\leqslant  C \, n^{-2M-2\alpha+1}\,  \Big(\frac{z}{n}\Big)^{-2\alpha-1} \left(\sqrt{ \Big(\frac{z}{n}\Big)^2+\frac{1}{\pi^2}} \ \right)^{-1}, \\
\Big| \, E^{n}_{y,M}\circ E^{n}_{x,M}f^{-\alpha} \Big|
&\leqslant  C \, n^{-4M-2\alpha}\, \Big(\frac{z}{n}\Big)^{-2\alpha-2}\left(\sqrt{ \Big(\frac{z}{n}\Big)^2+\frac{1}{\pi^2}} \ \right)^{-2}.
\end{align*} 
This proves in view of \eqref{e-trace-near} the statement.
 \end{proof}
 
\subsection{Estimates of the resolvent trace of $9$-point star Laplacian $\widetilde{\Delta}_n$}\label{emf-res-9} \ \medskip

\noindent In view of \eqref{resolvent-traces-discrete}, we simplify notation below by introducing 
\begin{equation}\label{g-definition}\begin{split}
g(x,y,n,z) &:=\frac{n^2}{\pi^2}\sin^2\Big(\frac{\pi x}{n}\Big)+\frac{n^2}{\pi^2}\sin^2\Big(\frac{\pi y}{n}\Big)
\\ &-\frac{2 n^2}{3\pi^2}\sin^2\Big(\frac{\pi x}{n}\Big)\sin^2\Big(\frac{\pi y}{n}\Big) + z^2.
\end{split}\end{equation}
By \eqref{resolvent-traces-discrete}, we have
\begin{align}
	\Tr(\widetilde{\Delta}_n+z^2)^{-\alpha} 
	= \sum_{k_1=0}^{n-1} \ \sum_{k_2=0}^{n-1} g(k_1,k_2,n,z)^{-\alpha}.
\end{align}

\noindent Now exactly the same argument as in Proposition \ref{trace-asymp} carries over
with $f$ replaced by $g$, and the only difference lies in the estimate \eqref{IE-EE-estimates}. 
There, the integral is replaced by (cf. the inequality in \eqref{ab-play})
\begin{align*}
	\int_0^1 \Big(\frac{\sin^2(\pi x)}{3\pi^2}+\frac{z^2}{n^2}\Big)^{-1}dx = \frac{\sqrt{3}}{\frac{z}{n}\sqrt{3(\frac{z}{n})^2+\frac{1}{\pi^2}}},
\end{align*}
where we used \eqref{trigonometry} again. Thus, we obtain the following counterpart 
to Proposition \ref{trace-asymp}. 

 \begin{Prop}\label{trace-asymp-9}
 	For any $z,\alpha>0$ and positive integers $M$ and $n$, the resolvent trace $\Tr(\widetilde{\Delta}_n+z^2)^{-\alpha} $ 
	has the representation 
 	\begin{equation}\label{e-trace-asymp-9}\begin{split}
 			\Tr(\widetilde{\Delta}_n+z^2)^{-\alpha}  &= n^{-2\alpha+2} \int_0^1 \int_0^1 g(x, y,1,z/n)^{-\alpha} dx dy \\
			&+ 2\, I^{n}_{y}\circ E^{n}_{x,M} g^{-\alpha} + E^{n}_{y,M}\circ E^{n}_{x,M}g^{-\alpha},
 	\end{split}\end{equation}
where for some constant $C>0$, depending only on $M$ and $\alpha$
\begin{equation}\label{error9}\begin{split}
\Big| \, I^{n}_{y}\circ E^{n}_{x,M} f^{-\alpha} \Big| 
&\leqslant  C \, n^{-2M-2\alpha+1}\,  \Big(\frac{z}{n}\Big)^{-2\alpha-1} \left(\sqrt{ 3\Big(\frac{z}{n}\Big)^2+\frac{1}{\pi^2}} \ \right)^{-1}, \\
\Big| \, E^{n}_{y,M}\circ E^{n}_{x,M}f^{-\alpha} \Big|
&\leqslant  C \, n^{-4M-2\alpha}\, \Big(\frac{z}{n}\Big)^{-2\alpha-2}\left(\sqrt{ 3\Big(\frac{z}{n}\Big)^2+\frac{1}{\pi^2}} \ \right)^{-2}.
\end{split}\end{equation} 
 \end{Prop}
 
\section{Asymptotic expansion of the discrete spectral zeta function}\label{s-asymp-zeta}

In this section we use Propositions \ref{trace-asymp} and \ref{trace-asymp-9} to 
derive a complete asymptotic expansion of the spectral zeta functions 
$\zeta(\Delta_n, s)$ and $\zeta(\widetilde{\Delta}_n,s)$ as $n\rightarrow \infty$
in the critical strip $\Rel(s) \in (0,\alpha/2)$. Recall the auxiliary functions \eqref{f-defn}
and \eqref{g-definition}
\begin{equation}\begin{split}
f(x,y,n,z) &:=\frac{n^2}{\pi^2}\sin^2\Big(\frac{\pi x}{n}\Big)+\frac{n^2}{\pi^2}\sin^2\Big(\frac{\pi y}{n}\Big) + z^2, \\
g(x,y,n,z) &:=\frac{n^2}{\pi^2}\sin^2\Big(\frac{\pi x}{n}\Big)+\frac{n^2}{\pi^2}\sin^2\Big(\frac{\pi y}{n}\Big)  + z^2
\\ &-\frac{2 n^2}{3\pi^2}\sin^2\Big(\frac{\pi x}{n}\Big)\sin^2\Big(\frac{\pi y}{n}\Big).
\end{split}\end{equation}

\subsection{Uniform series representation for $f^{-\alpha}$ and $g^{-\alpha}$}
Recall the notation for $f(x,y,n,z)$ from \eqref{f-defn}. We will use the 
series representation, derived here, decisively in the next subsection.

\begin{Lem}\label{F-series-lemma} We have the following series representation for any $N \in \N$
\begin{equation}\label{e-Taylor-series}\begin{split}
f(x,y,n,z)^{-\alpha} &=\sum_{m=0}^\infty  n^{-2m} \sum_{j=0}^m \frac{F_{m,j}(x,y)}{(x^2+y^2+z^2)^{\alpha+j}} \\
&=: \sum_{m=0}^{N-1}  n^{-2m} \sum_{j=0}^m \frac{F_{m,j}(x,y)}{(x^2+y^2+z^2)^{\alpha+j}} + n^{-2N} \frac{G \cdot (x^2+y^2)^N}{(x^2+y^2+z^2)^{\alpha}}, \\
g(x,y,n,z)^{-\alpha} &=\sum_{m=0}^\infty  n^{-2m} \sum_{j=0}^m \frac{\widetilde{F}_{m,j}(x,y)}{(x^2+y^2+z^2)^{\alpha+j}} \\
&=: \sum_{m=0}^{N-1}  n^{-2m} \sum_{j=0}^m \frac{\widetilde{F}_{m,j}(x,y)}{(x^2+y^2+z^2)^{\alpha+j}} + n^{-2N} \frac{\widetilde{G} \cdot (x^2+y^2)^N}{(x^2+y^2+z^2)^{\alpha}},
\end{split}\end{equation}
where,  $F_{m,j}$ and $\widetilde{F}_{m,j}$ are symmetric polynomials in $(x,y)$, homogeneous of order $2m+2j$. Moreover, if
$n \geqslant n_0$ sufficiently large, $G$ and $\widetilde{G}$ are bounded uniformly in $x,y,z \geqslant 0$.
\end{Lem}

\begin{proof} We prove the statement for $f$. The statement for $g$ follows along the same lines. 
Using $\sin^2 \theta = (1-\cos2 \theta)/2$ and the cosinus series, we obtain for $u \in \C$
\begin{equation}\label{sinus-taylor}
\frac{n^2}{\pi^2	} \sin^2\left(\frac{\pi u}{n}\right)=\sum_{k=1}^\infty \frac{(-1)^{k-1}}{(2k)!}\frac{2^{2k-1}\pi^{2k-2}}{n^{2k-2}}u^{2k}.
\end{equation}
Plugging this in $f^{-\alpha}(x,y,n,z)$ for $x$ and $y$, we obtain 
\begin{align*}
f(x,y,n,z)^{-\alpha}&= \Big(\sum_{k=1}^\infty \frac{(-1)^{k-1}}{(2k)!}\frac{2^{2k-1}\pi^{2k-2}}{n^{2k-2}}(x^{2k}+y^{2k})+z^2\Big)^{-\alpha}\\
&=(x^2+y^2+z^2)^{-\alpha}\Big(1-\sum_{k=2}^\infty \frac{(-1)^{k+1}}{(2k)!}\frac{2^{2k-1}\pi^{2k-2}}{n^{2k-2}}\frac{(x^{2k}+y^{2k})}{(x^2+y^2+z^2)}\Big)^{-\alpha} \\
&=(x^2+y^2+z^2)^{-\alpha}\Big(1-\sum_{k=0}^\infty \frac{(-1)^{k+1}}{(2k+4)!}\frac{2^{2k+3}\pi^{2k+2}}{n^{2k+2}}\frac{(x^{2k+4}+y^{2k+4})}{(x^2+y^2+z^2)}\Big)^{-\alpha}.
\end{align*}
We simplify notation by setting
$$
d_k := \frac{(-1)^{k+1}}{(2k+4)!} 2^{2k+3}\pi^{2k+2}.
$$
Then we find by the generalized binomial formula 
\begin{align*}
f(x,y,n,z)^{-\alpha} 
=(x^2+y^2+z^2)^{-\alpha}\sum_{j=0}^\infty {\alpha+j-1 \choose j}\Big(\sum_{k=0}^\infty d_k\frac{(x^{2k+4}+y^{2k+4})}{(x^2+y^2+z^2)}n^{-2k-2}\Big)^{j}.
\end{align*}
In order to evaluate the $j$-th power of the infinite sum above, we recall the general rule for the $j$-th Cauchy
product of an absolutely convergent series
\begin{align}\label{Cauchy-product}
\Big( \sum_{k=0}^\infty a_k \Big)^j = \sum_{k=0}^\infty \sum_{k_2=0}^{k} \dots \sum_{k_{j}=0}^{k_{j-1}} a_{k_j} 
\prod_{\ell = 1}^{j-1} a_{k_{\ell}-k_{\ell+1}}.
\end{align}
This leads to the following expression
\begin{align*}
{\alpha+j-1 \choose j}\Big(\sum_{k=0}^\infty d_k\frac{(x^{2k+4}+y^{2k+4})}{(x^2+y^2+z^2)}n^{-2k-2}\Big)^{j}= \sum_{k=0}^\infty \frac{F'_{k,j}(x,y)}{(x^2+y^2+z^2)^j}n^{-2(k+j)},
\end{align*}
where functions $F'_{k,j}$ are given in view of \eqref{Cauchy-product} by
	\begin{equation}\label{F}\begin{split}
	F'_{k,j}(x,y)={\alpha+j-1 \choose j}\sum_{k_2=0}^{k}\dots \sum_{k_{j}=0}^{k_{j-1}} d_{k_j} &\prod_{\ell = 1}^{j-1} d_{k_{\ell}-k_{\ell+1}}
	 \\ (x^{2k_j+4}+y^{2k_j+4}) &\prod_{\ell = 1}^{j-1} (x^{2(k_{\ell}-k_{\ell+1})+4}+y^{2(k_{\ell}-k_{\ell+1})+4}). 
	\end{split}\end{equation}
Note that each $F'_{k,j}(x,y)$ is homogeneous of order $2k+4j$ in $(x,y)$. \medskip

We also need an estimate of $F'_{k,j}$ uniformly in $x,y \geqslant 0$ and $k,j \in \N_0$. Note that the binomial factor in \eqref{F} can be estimated against $j^{\alpha -1}$
up to a uniform constant. Also, the number of summands in the Volterra series in \eqref{F} can be estimated against $\frac{k^{j-1}}{(j-1)!}$, up to 
a uniform constant. This yields 
	\begin{align*}
	\Big| F'_{k,j}(x,y) \Big| &\leq C_1 j^{\alpha -1} \frac{k^{j-1}}{(j-1)!}  \frac{2^{2k+3j}\pi^{2k+2j} 
	\Bigl(x^2+y^2\Bigr)^{k+2j}}{\Gamma(2k_j+5)\prod\limits_{\ell = 1}^{j-1} \Gamma(2 (k_{\ell}-k_{\ell+1})+5)} 
	\\ &\leq C_1 j^{\alpha -1} \frac{k^{j-1}}{(j-1)!}  \frac{2^{2k+3j}\pi^{2k+2j} \Bigl(x^2+y^2\Bigr)^{k+2j} }{\Bigl( \Gamma \left( 2 k/j +5\right) \Bigr)^j},
	\end{align*}
for some uniform constant $C_1>0$. Here, we have used in the second step the logarithmic convexity and Jensen's inequality to estimate
the product of Gamma functions in the denominator. Studying the behaviour of the individual terms as 
$k\leqslant j \to \infty$, we conclude that for yet another uniform constant $C_2>0$, sufficiently large
\begin{align}\label{individual-F}
	\Big| F'_{k,j}(x,y) \Big| &\leq C_2^j \Bigl(x^2+y^2\Bigr)^{k+2j}.
\end{align}
Summarizing, we find after some reshuffling for any $N \in \N$
\begin{equation}\begin{split}
f(x,y,n,z)^{-\alpha} &=\sum_{m=0}^\infty  n^{-2m} \sum_{k+j = m} \frac{F'_{k,j}(x,y)}{(x^2+y^2+z^2)^{\alpha+j}} \\
&= \sum_{m=0}^{N-1}  n^{-2m} \sum_{j=0}^m \frac{F'_{m-j,j}(x,y)}{(x^2+y^2+z^2)^{\alpha+j}} + n^{-2N} \frac{G \cdot (x^2+y^2)^N}{(x^2+y^2+z^2)^{\alpha}},
\end{split}\end{equation}
where, provided $n>C_2$, the term $G$ in the nominator is estimated by \eqref{individual-F} 
\begin{align}\label{G-estimate}
|G| \leqslant \sum_{m=N}^{\infty} C_2^m n^{-2(m-N)} \sum_{j=0}^m \frac{(x^2+y^2)^j}{(x^2+y^2+z^2)^{j}} \leqslant C,
\end{align}
for some uniform constant $C>0$, independent of $n$ and $z$. This proves the statement
if we set $F_{m,j}:= F'_{m-j,j}$.
\end{proof}

\begin{Rem} The coefficients $F_{m,j}$ and $\widetilde{F}_{m,j}$ arise from the 
Taylor expansion \eqref{sinus-taylor} and are explicit from 
\eqref{F}. For instance we have for any $x,y \geq 0$ and $\alpha = 2$
\begin{equation}\label{explicit-F}\begin{split}
&F_{0,0} (x,y) = \widetilde{F}_{0,0} (x,y) = 1, \\
&F_{1,0} (x,y)=0, \quad F_{1,1} (x,y) = \frac{2}{3} \pi^2 \Bigl( x^4 + y^4 \Bigr), \\
&\widetilde{F}_{1,0} (x,y)=0, \quad  \widetilde{F}_{1,1} (x,y) = \frac{2}{3} \pi^2 \Bigl( x^2 + y^2 \Bigr)^2.
\end{split}\end{equation}
\end{Rem}

\subsection{Asymptotic expansion of the spectral zeta functions}
As in \eqref{e-zeta-reg-int} we write
\begin{equation}\label{det-formula}\begin{split}
	\zeta(\Delta_n, s) &= V_\alpha(s)\rint_0^\infty z^{2\alpha-2s-1}\Tr(\Delta_n+z^2)^{-\alpha} dz, \\
	\zeta(\widetilde{\Delta}_n, s) &= V_\alpha(s)\rint_0^\infty z^{2\alpha-2s-1}\Tr(\widetilde{\Delta}_n+z^2)^{-\alpha} dz.
\end{split}\end{equation}
We proceed below with studying the $5$-point star Laplacian $\Delta_n$. 
The estimates for the $9$-point star Laplacian $\widetilde{\Delta}_n$ follows along the same lines.
We split the regularized integral in two parts, namely 
 \begin{align*}
 	\rint_0^\infty z^{2\alpha-2s-1} \Tr(\Delta_n+z^2)^{-\alpha}  \, dz
	= \rint_0^n + \rint_n^\infty
	=:J_\alpha^0(s)+J_\alpha^\infty(s).
 \end{align*}
Below, Proposition \ref{J-inf} derives the asymptotics of $J_\alpha^\infty$. 
Proposition \ref{J_zero} discusses the asymptotics of $J_\alpha^0$.
Both results use the (reduced) discrete resolvent trace asymptotics in Proposition \ref{trace-asymp}.
The corresponding result for the $9$-point star Laplacian is discussed in the 
next subsection. 

\begin{Prop}\label{J-inf}
For $\Rel(s)>0$ we have as $n\rightarrow \infty$
\begin{equation}
	J_\alpha^\infty(s)\sim a_\infty(s)n^{\alpha-2s} + O(n^{-\infty}),
\end{equation}	
with the coefficient given by 
\begin{equation}
	a_\infty(s)= \int_1^\infty z^{2\alpha-2s-1}\int_0^1\int_0^1 \Big(\frac{\sin^2(\pi x)}{\pi^2}+\frac{\sin^2(\pi y)}{\pi^2}+z^2\Big)^{-\alpha} dxdydz.
\end{equation}
\end{Prop}

 \begin{proof}
Let $M\in\N$. We plug in \eqref{e-trace-asymp} and \eqref{error5} into $J_\alpha^\infty(s)$, which gives
\begin{align*}
J_\alpha^\infty(s) &= n^{-2\alpha+2} \rint_n^\infty \int_0^1 \int_0^1 z^{2\alpha-2s-1} \Big(\frac{\sin^2(\pi x)}{\pi^2}+\frac{\sin^2(\pi y)}{\pi^2}+\frac{z^2}{n^2}\Big)^{-\alpha}  dx dy dz \\
&+ n^{-2M-2\alpha+1} \rint_n^\infty z^{2\alpha-2s-1} \,  O \left( \Big(\frac{z}{n}\Big)^{-2\alpha-1} \left(\sqrt{ \Big(\frac{z}{n}\Big)^2+\frac{1}{\pi^2}} \ \right)^{-1} \right) dz \\
&+ n^{-4M-2\alpha} \rint_n^\infty z^{2\alpha-2s-1} O \left( \Big(\frac{z}{n}\Big)^{-2\alpha-2}\left(\sqrt{ \Big(\frac{z}{n}\Big)^2+\frac{1}{\pi^2}} \ \right)^{-2} \right) dz.
\end{align*}
Here, the absolute value of each $O$-term may be estimated against the term in its brackets, up to 
a constant, depending only on $M$. For a fixed $n\in \N$, each integrand 
is $O(z^{-2s-1})$ as $z\to \infty$. Hence, for $\Rel(s)>0$, the regularized integrals exist in the 
usual sense and we obtain after a change of coordinates $z \mapsto zn$
\begin{align*}
J_\alpha^\infty(s) &= n^{2-2s} \int_1^\infty z^{2\alpha-2s-1} \int_0^1 \int_0^1 \Big(\frac{\sin^2(\pi x)}{\pi^2}+\frac{\sin^2(\pi y)}{\pi^2}+z^2\Big)^{-\alpha}  dx dy dz \\
&+ n^{-2M-2s+1} O \left(  \int_1^\infty z^{2\alpha-2s-1} \,  z^{-2\alpha-1} \left(\sqrt{ z^2+\frac{1}{\pi^2}} \ \right)^{-1}  dz \right) \\
&+ n^{-4M-2s} O \left( \int_1^\infty z^{2\alpha-2s-1}  z^{-2\alpha-2}\left(\sqrt{ z^2+\frac{1}{\pi^2}} \ \right)^{-2} dz  \right) ,
\end{align*}
Since $M \in \N$ was arbitrary, taking $M \to \infty$, proves the statement. 
\end{proof}
 	
In contrast to the discussion of $J_\alpha^\infty(s)$ in Propositon \ref{J-inf} above, 
the asymptotic analysis of $J_\alpha^0(s)$ is much more intricate, since the 
integrals exist only in the Hadamard regularized sense. 

\begin{Prop}\label{J_zero}
For $\Rel(s) < \alpha/2$ we have for any integer $M \in \N$ as $n\to \infty$	
\begin{equation}
J_\alpha^0(s)\sim a_0(s)n^{\alpha-2s}+\sum_{m=0}^{M-1} b_m(s)n^{-2m} 
+ O(n^{-2M -2s +1}).
\end{equation}
The leading coefficient is explicitly given by 
\begin{equation}\label{a0}
a_0(s) =  \int_0^1 z^{2\alpha-2s-1}  \int_0^1 \int_0^1 
\Big(\frac{\sin^2(\pi x)}{\pi^2}+\frac{\sin^2(\pi y)}{\pi^2}+z^2\Big)^{-\alpha}  dx dy dz
\end{equation}
The higher order coefficients are given by 	
\begin{equation}\label{b0}\begin{split}
b_j(s) &=\frac{4}{(2M+1)!)^2}\rint_0^\infty z^{2\alpha-2s-1}\int_0^\infty \int_0^\infty B_{2M+1}(x-[x])B_{2M+1}(y-[y])\\[2mm]
& \quad \partial_x^{(2M+1)}\partial_y^{(2M+1)} 
\sum_{j=0}^m \frac{F_{m,j}(x,y)}{(x^2+y^2+z^2)^{\alpha+j}} \\[2mm]
&+\frac{8}{(2M+1)!}\rint_0^\infty z^{2\alpha-2s-1}\int_0^\infty \int_0^\infty B_{2M+1}(x-[x]) \\[2mm] &\quad \partial_x^{(2M+1)} 
\sum_{j=0}^m \frac{F_{m,j}(x,y)}{(x^2+y^2+z^2)^{\alpha+j}}  dxdydz
\end{split}\end{equation}
\end{Prop}

\begin{proof} Recall the Euler Maclaurin formula \eqref{e-emf-op} applied to the resolvent trace
\begin{equation}\label{trace-EMLF}\begin{split}
\Tr(\Delta_n+z^2)^{-\alpha}  &= \sum_{k_1 =0}^{n-1} \ \sum_{k_2 =0}^{n-1} f(k_1,k_2,n,z)^{-\alpha} \\
&= \Bigg( I^{n}_{y} \circ I^{n}_{x} + 
2 I^{n}_{y}\circ E^{n}_{x,M} + E^{n}_{y,M} \circ E^{n}_{x,M} \Bigg) 
f(x,y,n,z)^{-\alpha}.
\end{split} \end{equation}
For a function $u \in C^\infty [0,n]$, such that 
$u(x) = u(n-x)$, we compute
\begin{align*}
&\int_{0}^n B_{2M+1}(x-[x]) \, \partial_{x}^{2M+1} u(x) dx \\
&= \int_0^{\frac{n}{2}} B_{2M+1}(x-[x]) \, \partial_{x}^{2M+1} u(x) dx
+ \int_{\frac{n}{2}}^n B_{2M+1}(x-[x]) \, \partial_{x}^{2M+1} u(n-x) dx \\
&= \int_0^{\frac{n}{2}} B_{2M+1}(x-[x]) \, \partial_{x}^{2M+1} u(x) dx
+ \int_0^{\frac{n}{2}} B_{2M+1}(1-(y-[y])) \, (-\partial_{y})^{2M+1} u(y) dy \\
&= 2\int_{0}^{\frac{n}{2}}  B_{2M+1}(x-[x]) \, \partial_{x}^{2M+1} u(x) dx,
\end{align*}
where in the last step we used symmetry of the Bernoulli polynomials $B_p(1-t) =(-1)^p B(t)$
for $t \geqslant 0$. Using symmetry of $f(x,y,n,z)$ and the computation above, we 
can rewrite \eqref{trace-EMLF} as follows
\begin{align*}
\Tr(\Delta_n+z^2)^{-\alpha}  = \Bigg( I^{n}_{y} \circ I^{n}_{x} + 
4 \Bigl( 2 I^{\frac{n}{2}}_{y}\circ E^{\frac{n}{2}}_{x,M} + E^{\frac{n}{2}}_{y,M} \circ E^{\frac{n}{2}}_{x,M} \Bigr) \Bigg) 
f(x,y,n,z)^{-\alpha}.
\end{align*}
Now we plug in the expansion in Lemma \ref{F-series-lemma} with 
\begin{align*}
&N:=M+1, \ \textup{for the expansion of} \quad I^{\frac{n}{2}}_{y}\circ E^{\frac{n}{2}}_{x,M}f^{-\alpha}, \\
&N':=2(M+1), \ \textup{for the expansion of} \quad  E^{\frac{n}{2}}_{y,M} \circ E^{\frac{n}{2}}_{x,M} f^{-\alpha}.
\end{align*} 
We obtain
\begin{equation}\label{big-expression-J}\begin{split}
&\Tr(\Delta_n+z^2)^{-\alpha}  =  \int_0^n \int_0^n
\Bigg(\frac{n^2}{\pi^2}\sin^2\Big(\frac{\pi x}{n}\Big)+\frac{n^2}{\pi^2}\sin^2\Big(\frac{\pi y}{n}\Big)+z^2\Bigg)^{-\alpha}  dx dy
\\ &+  8 \,  I^{\frac{n}{2}}_{y}\circ E^{\frac{n}{2}}_{x,M} \Bigg( \sum_{m=0}^{N-1}  n^{-2m} \sum_{j=0}^m \frac{F_{m,j}(x,y)}{(x^2+y^2+z^2)^{\alpha+j}}  
+ n^{-2N}\, \frac{G \cdot (x^2+y^2)^N}{(x^2+y^2+z^2)^{\alpha}} \Bigg) 
\\ &+  4  \, E^{\frac{n}{2}}_{y,M} \circ E^{\frac{n}{2}}_{x,M} \Bigg( \sum_{m=0}^{N'-1}  n^{-2m} \sum_{j=0}^m \frac{F_{m,j}(x,y)}{(x^2+y^2+z^2)^{\alpha+j}}  
+ n^{-2N'} \, \frac{G' \cdot (x^2+y^2)^{N'}}{(x^2+y^2+z^2)^{\alpha}} \Bigg).
\end{split} \end{equation}
Plug this into \eqref{det-formula} and simplify notation by writing for the first term in \eqref{big-expression-J}
\begin{align*}
I(s,n) &:= \rint_0^n z^{2\alpha-2s-1}  \int_0^n \int_0^n 
\Bigg(\frac{n^2}{\pi^2}\sin^2\Big(\frac{\pi x}{n}\Big)+\frac{n^2}{\pi^2}\sin^2\Big(\frac{\pi y}{n}\Big)+z^2\Bigg)^{-\alpha} dx dy dz.
\end{align*}
For the individual summands in the second and third lines in \eqref{big-expression-J} we set
	\begin{align*}
		W_{m,j}^1(s,n) &:= \rint_0^nz^{2\alpha-2s-1} \Bigl( I^{\frac{n}{2}}_{y}\circ E^{\frac{n}{2}}_{x,M}\Bigr)
		\frac{F_{m,j}(x,y)}{(x^2+y^2+z^2)^{\alpha+j}}dz, \\
		W^2_{m,j}(s,n) &:=\rint_0^nz^{2\alpha-2s-1} \Bigl(E^{\frac{n}{2}}_{y,M} \circ E^{\frac{n}{2}}_{x,M}  \Bigr) 
		\frac{F_{m,j}(x,y)}{(x^2+y^2+z^2)^{\alpha+j}}dz.
	\end{align*}
For the other remaining summands in \eqref{big-expression-J} (plugged into \eqref{det-formula}) we set
	\begin{align*}
		R_N^1(s,n) &:= \rint_0^n z^{2\alpha-2s-1} \Bigl( I^{\frac{n}{2}}_{y}\circ E^{\frac{n}{2}}_{x,M} \Bigr)  
		\frac{G \cdot (x^2+y^2)^N}{(x^2+y^2+z^2)^{\alpha}} dz,\\
		R^2_{N'}(s,n) &:= \rint_0^n z^{2\alpha-2s-1}  \Bigl( E^{\frac{n}{2}}_{y,M} \circ E^{\frac{n}{2}}_{x,M}  \Bigr) 
		\frac{G' \cdot (x^2+y^2)^{N'}}{(x^2+y^2+z^2)^{\alpha}} dz.
	\end{align*}
This yields in view of \eqref{big-expression-J} the following expression
\begin{align*}
J_\alpha^0(s) = I(s,n) + 8 \Bigg( \sum_{m=0}^{N-1}  n^{-2m} \sum_{j=0}^m W_{m,j}^1(s,n) + n^{-2N} R_N^1(s,n) \Bigg) \\
+ 4 \Bigg( \sum_{m=0}^{N'-1}  n^{-2m} \sum_{j=0}^m W_{m,j}^1(s,n) + n^{-2N'} R_{N'}^2(s,n) \Bigg).
\end{align*}
The rest of the proof proceeds with estimating the individual terms. 
For the term $I(s,n)$ note that by \eqref{explicit-integral} the $dz$ integrand is $O(z^{\alpha-2s-1})$ as $z\to 0$.
Hence for $\Rel(s) < \alpha/2$, the regularized integral exists in the usual sense and we obtain after
a change of variables
\begin{align}\label{a0-inside}
I(s,n) = n^{2-2s} \int_0^1 z^{2\alpha-2s-1}  \int_0^1 \int_0^1 
\Big(\frac{\sin^2(\pi x)}{\pi^2}+\frac{\sin^2(\pi y)}{\pi^2}+z^2\Big)^{-\alpha}  dx dy dz.
\end{align}
This proves \eqref{a0}. \medskip

\noindent For the other terms we estimate for any $x,y,z > 0$ and a uniform constant $C>0$, 
using that $F_{m,j}$ is a polynomial and any derivative just lowers the exponent 
in \eqref{individual-F} by one half\footnote{The order of differentiation in $x,y$ is 
odd, so that the additional $x,y$ appears in the estimates.}

\begin{equation}\label{F-derivative-estimates}\begin{split}
		\left| \partial_x^{\!(2M+1) \atop } \! \! \!  \! \! \! \frac{F_{m,j}(x,y)}{(x^2+y^2+z^2)^{\alpha+j}} \right| \! &\leqslant C 
		\left\{ \begin{split}
		 & \frac{x(x^2+y^2)^{m-M}}{(x^2+y^2+z^2)^{\alpha+1}},  \textup{for} \ m \geqslant M,
		  \\ & x (x^2+y^2+z^2)^{-\alpha-M+m-1}, m \leqslant M-1, \end{split} \right. \\[2mm]
		\left| (\partial_y \partial_x)^{\!(2M+1) \atop }  \! \! \!  \! \! \! \frac{F_{m,j}(x,y)}{(x^2+y^2+z^2)^{\alpha+j}} \right| \! \! &\leqslant C
		\left\{ \begin{split}
		 & \frac{xy(x^2+y^2)^{m-2M}}{(x^2+y^2+z^2)^{\alpha+2}}, \textup{for} \ m \geqslant 2(M+1),
		\\ & \frac{xy}{(x^2+y^2+z^2)^{\alpha+2M-m+2}}, \ m \leqslant 2M+1.\end{split} \right.
	\end{split}\end{equation}

\noindent For the estimate of the error term  $R_N^1(s,n)$, we note using \eqref{F-derivative-estimates} and $N=M+1$
		\begin{equation}\label{ableitungen}\begin{split}
			\partial_x^{2M+1} \frac{G \cdot (x^2+y^2)^N}{(x^2+y^2+z^2)^{\alpha}}
			= \, &\sum_{m=N}^\infty n^{-2(m-N)} \sum_{j=0}^m \partial_x^{2M+1}  \frac{F_{m,j}(x,y)}{(x^2+y^2+z^2)^{\alpha+j}} \\
		        = \, &\sum_{m=N}^\infty n^{-2(m-N)} \sum_{j=0}^m O\Bigg(\frac{x(x^2+y^2)^{m-M}}{(x^2+y^2+z^2)^{\alpha+1}}\Bigg) \\
		        = \, &\sum_{m=N}^\infty \frac{(x^2+y^2)^{m-N}}{n^{2(m-N)}} \, O\Bigg(\frac{x(x^2+y^2)}{(x^2+y^2+z^2)^{\alpha+1}}\Bigg) \\
		        = \, & O\Bigl(x(x^2+y^2+z^2)^{-\alpha}\Bigr),
		\end{split}\end{equation}
for $x,y \in [0,n/2]$, such that the series converge absolutely and the $O$ constant is uniform in $(n,z)$. 
From here we conclude for $z \in [0,n], n \geq 2$ and a constant $C>0$, 
depending only on $M$
	\begin{align*}
		\left| \  I^{\frac{n}{2}}_{y}\circ E^{\frac{n}{2}}_{x,M} \frac{G \cdot (x^2+y^2)^N}{(x^2+y^2+z^2)^{\alpha}} \ \right| 
		&\leqslant C  \int_0^\frac{n}{2} (n^2+y^2+z^2)^{-\alpha+1} dy + C  \int_0^\frac{n}{2} (y^2+z^2)^{-\alpha+1} dy \\ 
		&\leqslant C  \int_0^\frac{n}{2} (n^2+y^2+z^2)^{-\alpha+1} dy + C  \int_0^1 (y+z^2)^{-\alpha+1} dy  \\ &+ C  \int_1^\frac{n}{2} (1+z^2)^{-\alpha+1} dy 
		\leqslant C \log n.
	\end{align*}
\noindent Similarly, we obtain for $R_{N'}^2(s,n)$, using  \eqref{F-derivative-estimates}  and $N'=2(M+1)$
		\begin{equation}\label{ableitungen2}\begin{split}
			&\partial_y^{2M+1}\partial_x^{2M+1} \frac{G' \cdot (x^2+y^2)^{N'}}{(x^2+y^2+z^2)^{\alpha}}
			= \sum_{m=N'}^\infty n^{-2(m-N')} \sum_{j=0}^m \partial_y^{2M+1} \partial_x^{2M+1}  \frac{F_{m,j}(x,y)}{(x^2+y^2+z^2)^{\alpha+j}} \\
		        &= \sum_{m=N'}^\infty  n^{-2(m-N')} \sum_{j=0}^m O\Bigg(\frac{xy(x^2+y^2)^{m-2M}}{(x^2+y^2+z^2)^{\alpha+2}}\Bigg)
		        =  O\Bigl(xy(x^2+y^2+z^2)^{-\alpha}\Bigr),
		\end{split}\end{equation}
for $x,y \in [0,n/2]$, where the $O$ constant is as before uniform in $(n,z)$. 
From here we conclude for $z \in [0,n]$ and a constant $C>0$, 
depending only on $M$
	\begin{align*}
		&\left| \  E^{\frac{n}{2}}_{y,M}\circ E^{\frac{n}{2}}_{x,M} \frac{G' \cdot (x^2+y^2)^{N'}}{(x^2+y^2+z^2)^{\alpha}} \ \right| \leqslant 
		C \log (x^2+y^2+z^2) \Big|_{x=0}^{x=\frac{n}{2}} \Big|_{y=0}^{y=\frac{n}{2}} \leqslant C \log n.
	\end{align*}
Plugging these estimates into the expressions for 
$R_N^1(s,n)$ and $R_{N'}^2(s,n)$, and noting that for $\Rel(s) < \alpha/2$ the $z$-integrals exist in the usual sense, 
we find
\begin{equation}\label{R-estimate}\begin{split}
		|R_N^1(s,n)|&\leqslant C \log(n) \int_0^n z^{2\alpha-2s-1} \leq C, \ \textup{for} \ N = M+1, \\
		|R_{N'}^2(s,n)|&\leqslant C \log(n) \int_0^n z^{2\alpha-2s-1} \leq C, \ \textup{for} \ N'=2(M+1).
\end{split}\end{equation}
It remains to study $W_{m,j}^1(s,n)$ and $W_{m,j}^2(s,n)$. We write
\begin{equation}\label{WWW}\begin{split}
		W_{m,j}^1(s,n) 
		= \Bigg( \rint_0^\infty -  \rint_n^\infty \Bigg) \frac{z^{2\alpha-2s-1}}{(2M+1)!}
		\Bigg( \int_{0}^\infty \int_{0}^\infty - \, \int_{\frac{n}{2}}^\infty \int_{\frac{n}{2}}^\infty - 
		\int_{0}^{\frac{n}{2}} \int_{\frac{n}{2}}^\infty - \int_{\frac{n}{2}}^\infty \int_{0}^{\frac{n}{2}} \Bigg) \\ 
		B_{2M+1}(x-[x]) \, \partial_{x}^{2M+1} \frac{F_{m,j}(x,y)}{(x^2+y^2+z^2)^{\alpha+j}} dx dy dz, \\
		W_{m,j}^2(s,n) 
		= \Bigg( \rint_0^\infty -  \rint_n^\infty \Bigg) \frac{z^{2\alpha-2s-1}}{(2M+1)!}
		\Bigg( \int_{0}^\infty \int_{0}^\infty - \, \int_{\frac{n}{2}}^\infty \int_{\frac{n}{2}}^\infty - 
		\int_{0}^{\frac{n}{2}} \int_{\frac{n}{2}}^\infty - \int_{\frac{n}{2}}^\infty \int_{0}^{\frac{n}{2}} \Bigg)  \\ 
		B_{2M+1}(y-[y])B_{2M+1}(x-[x]) \, \partial_{y}^{2M+1}\partial_{x}^{2M+1} \frac{ F_{m,j}(x,y)}{(x^2+y^2+z^2)^{\alpha+j}} dx dy dz.
\end{split}\end{equation}
We want to use \eqref{F-derivative-estimates} to estimate the individual integrals in \eqref{WWW}: e.g.
we find after change of variables
\begin{align*}
&\rint_0^\infty \int_0^\infty \int_{\frac{n}{2}}^\infty z^{2\alpha-2s-1} x (x^2+y^2+z^2)^{-\alpha-M+m-1} \, dx dy dz
\\ &= \int_0^\infty \int_0^\infty z^{2\alpha-2s-1} (\alpha+M-m)^{-1} (n^2/\, 4+y^2+z^2)^{-\alpha-M+m} \, dy dz
\\ &= (n/2)^{-2M+2m-2s+1} \int_0^\infty \int_0^\infty z^{2\alpha-2s-1} (\alpha+M-m)^{-1} (1+y^2+z^2)^{-\alpha-M+m} \, dy dz
\\ &= O \Bigl(n^{-2M+2m-2s+1} \Bigr), \quad \textup{as} \ n\to \infty.
\end{align*}
Similarly, we compute
\begin{align*}
&\rint_n^\infty \int_0^\infty \int_0^\infty z^{2\alpha-2s-1} x (x^2+y^2+z^2)^{-\alpha-M+m-1} \, dx dy dz
\\ &= \int_n^\infty \int_0^\infty z^{2\alpha-2s-1} (\alpha+M-m)^{-1} (y^2+z^2)^{-\alpha-M+m} \, dy dz
\\ &= n^{-2s -2M+2m+1} \int_1^\infty \int_0^\infty z^{2\alpha-2s-1} (\alpha+M-m)^{-1} (y^2+z^2)^{-\alpha-M+m} \, dy dz
\\ &= O \Bigl(n^{-2M+2m-2s+1} \Bigr), \quad \textup{as} \ n\to \infty.
\end{align*}
Similar estimates hold for other integrals where
at least one of the variables $(x,y,z)$ is integrated over $[n,\infty)$ or $[n/2,\infty)$. In all of those cases, 
the $z$-integral exists in the usual sense and hence \eqref{F-derivative-estimates} may be applied. 
We conclude
\begin{equation*}\begin{split}
		W_{m,j}^1(s,n) 
		&= \rint_0^\infty \frac{z^{2\alpha-2s-1}}{(2M+1)!}
		 \int_0^\infty \int_0^\infty B_{2M+1}(x-[x]) \\ & \partial_{x}^{2M+1} \frac{F_{m,j}(x,y)}{(x^2+y^2+z^2)^{\alpha+j}} dx dy dz
		 + O \Bigl(n^{-2M+2m-2s+1} \Bigr), \quad \textup{as} \ n\to \infty, \\
		W_{m,j}^2(s,n) 
		&=  \rint_0^\infty \frac{z^{2\alpha-2s-1}}{(2M+1)!}
		\int_0^\infty \int_0^\infty B_{2M+1}(y-[y])B_{2M+1}(x-[x]) \\ & \partial_{y}^{2M+1}\partial_{x}^{2M+1}
		 \frac{F_{m,j}(x,y)}{(x^2+y^2+z^2)^{\alpha+j}} dx dy dz + 
		 O \Bigl(n^{ -4M+2m-2s+1} \Bigr), \quad \textup{as} \ n\to \infty.
\end{split}\end{equation*}
This, together with \eqref{R-estimate} and \eqref{a0-inside} proves the statement.
\end{proof}

Propositions \ref{J-inf} and \ref{J_zero}, as well as their analogous 
estimates for the $9$-point star Laplacian $\widetilde{\Delta}_n$, combine to the main result of this section.

\begin{Thm}\label{asymp-zeta}
For $\Rel(s) \in (0, \alpha/2)$ we have for any integer $M \in \N$ as $n\to \infty$	
\begin{equation} \begin{split}
&\zeta(\Delta_n, s) = V_\alpha(s) \Bigg( a(s) \, n^{\alpha-2s}+\sum_{m=0}^{M-1} b_m(s) \, n^{-2m} \Bigg)
+ O(n^{-2M -2s +2}), \\
&\zeta(\widetilde{\Delta}_n, s) = V_\alpha(s) \Bigg( \widetilde{a}(s) \, n^{\alpha-2s}+\sum_{m=0}^{M-1} \widetilde{b}_m(s) \, n^{-2m} \Bigg)
+ O(n^{-2M -2s +2}).
\end{split}\end{equation}
The leading coefficients $a(s)$ and $\widetilde{a}(s)$ are explicitly given by 
\begin{equation}\label{a0-theorem}\begin{split}
a(s) & =  \int_0^\infty z^{2\alpha-2s-1}  \int_0^1 \int_0^1 
\Big(\frac{\sin^2(\pi x)}{\pi^2}+\frac{\sin^2(\pi y)}{\pi^2}+z^2\Big)^{-\alpha}  dx dy dz, \\
\widetilde{a}(s) &=  \int_0^\infty z^{2\alpha-2s-1}  \int_0^1 \int_0^1 
\Bigg( \frac{\sin^2(\pi x)}{\pi^2}+\frac{\sin^2(\pi y)}{\pi^2} + z^2 \\  &-\frac{2 n^2}{3\pi^2}\sin^2\Big(\frac{\pi x}{n}\Big)\sin^2\Big(\frac{\pi y}{n}\Big) \Bigg)^{-\alpha}  dx dy dz.
\end{split}\end{equation}
The higher order coefficients are given by 	
\begin{equation}\label{b0-theorem}\begin{split}
b_m(s) &=  
\rint_0^\infty z^{2\alpha-2s-1} 4 \left(E_{y,M}^{\infty}\circ E_{x,M}^{\infty}+2\cdot I_{y}^{\infty} \circ 
E_{x,M}^{\infty}\right) \sum_{j=0}^m \frac{F_{m,j}(x,y)}{(x^2+y^2+z^2)^{\alpha+j}} dz\\
&= \frac{4}{(2M+1)!)^2}\rint_0^\infty z^{2\alpha-2s-1}\int_0^\infty \int_0^\infty B_{2M+1}(x-[x])B_{2M+1}(y-[y])\\[2mm]
& \quad \partial_x^{(2M+1)}\partial_y^{(2M+1)} 
\sum_{j=0}^m \frac{F_{m,j}(x,y)}{(x^2+y^2+z^2)^{\alpha+j}} \\[2mm]
&+\frac{8}{(2M+1)!}\rint_0^\infty z^{2\alpha-2s-1} \int_0^\infty \int_0^\infty B_{2M+1}(x-[x]) \\[3mm] &\quad \partial_x^{(2M+1)} 
\sum_{j=0}^m \frac{F_{m,j}(x,y)}{(x^2+y^2+z^2)^{\alpha+j}}  dxdydz,
\end{split}\end{equation}
where $F_{m,j}(x,y)$ is defined in Lemma \ref{F-series-lemma}.
The expression for $\widetilde{b}_m(s)$ is exactly the same, with 
$F_{m,j}(x,y)$ replaced by $\widetilde{F}_{m,j}(x,y)$, as defined in Lemma \ref{F-series-lemma}.
\end{Thm}

\section{Computing the coefficients in the asymptotic expansion}\label{asymp-coeff} 

In this section, we will express the coefficients $b_0(s), b_1(s)$ as well as 
$\widetilde{b}_0(s), \widetilde{b}_1(s)$ in Theorem \ref{asymp-zeta}
explicitly in terms of the spectral zeta function on the torus, $\zeta(\Delta,s)$ in \eqref{e-zeta-reg-int}, and related functions. 

\begin{Prop}\label{first-coeff}
Consider  the Laplace Beltrami operator $\Delta$ on the two-dimensional torus 
$\mathbb{T}^2$ and its spectral zeta function $\zeta(\Delta,s)$
The coefficients $b_0(s), \widetilde{b}_0(s)$ in Theorem \ref{asymp-zeta} are given by 
\begin{equation}
b_0(s) = \widetilde{b}_0(s) =  V_\alpha(s)^{-1} \zeta(\Delta,s).
\end{equation}
\end{Prop}
\begin{proof} As asserted in \eqref{explicit-F}, we have
\begin{equation*}
F_{0,0} (x,y) = \widetilde{F}_{0,0} (x,y) = 1.
\end{equation*}
Let us write $I_{x_i}^{\R}$ and $E_{x_i,M}^{\R}$ for the operators in \eqref{IADE},
where the integration region is replaced by $\R$. Then, in view of \eqref{b0-theorem} and the 
symmetry of the integrands around $x=0$ and $y=0$, we obtain
\begin{equation}\label{RRR}\begin{split}
&b_0(s) =  \widetilde{b}_0(s) \\ &= \rint_0^\infty z^{2\alpha-2s-1} 
4 \left(E_{y,M}^{\infty}\circ E_{x,M}^{\infty}+2\cdot I_{y}^{\infty} \circ E_{x,M}^{\infty}\right) (x^2+y^2+z^2)^{-\alpha} dz
\\ &=\rint_0^\infty z^{2\alpha-2s-1}  \left(E_{y,M}^{\R}\circ E_{x,M}^{\R}+2\cdot I_{y}^{\R} \circ E_{x,M}^{\R}\right) (x^2+y^2+z^2)^{-\alpha} dz.
\end{split}\end{equation}
By the Euler-Maclaurin formula \eqref{e-emf-op} (with $n \to \infty$) we conclude
\begin{align*}
b_0(s) &=  \widetilde{b}_0(s) = \rint_0^\infty z^{2\alpha-2s-1} \sum_{k_1=-\infty}^\infty\sum_{k_1=-\infty}^\infty (k_1^2+k_2^2+z^2)^{-\alpha} dz \\ &- 	
\rint_0^\infty z^{2\alpha-2s-1} \int_\R \int_\R (x^2+y^2+z^2)^{-\alpha} dx dy dz \\ 
&= \rint_0^\infty z^{2\alpha-2s-1} \textup{Tr} \Bigl( \Delta + z^2 \Bigr)^{-\alpha} dz - \rint_0^\infty z^{2\alpha-2s-1} \int_\R \int_\R (x^2+y^2+z^2)^{-\alpha} dx dy dz,
\end{align*}
where $\Delta$ is the Laplace Beltrami operator on the two-dimensional torus $\mathbb{T}^2$.
Note that $(x^2+y^2+z^2)^{-\alpha}$ is integrable on $\R^2$.
As computed e.g. in \cite[(1.13c)]{Les:DOR}, the regularized integral over $[0,\infty)$ 
for any power of $z$ vanishes and hence we conclude
\begin{equation}\label{reg-zero}\begin{split}
&\rint_0^\infty z^{2\alpha-2s-1} \int_\R \int_\R (x^2+y^2+z^2)^{-\alpha}dx dy dz \\
&= \rint_0^\infty z^{-2s+1} dz \cdot \int_\R \int_\R  (x^2+y^2+1)^{-\alpha}dx dy = 0.
\end{split}\end{equation}
Thus we arrive at the following expression
\begin{align*}
V_\alpha (s) b_0(s) = V_\alpha (s) \widetilde{b}_0(s)
= V_\alpha(s) \rint_0^\infty z^{2\alpha-2s-1}  \textup{Tr} \Bigl( \Delta + z^2 \Bigr)^{-\alpha} dz=\zeta( \Delta, s),
\end{align*}
where we used the representation \eqref{e-zeta-reg-int} and \eqref{VA}.
\end{proof}
	
\noindent We proceed with the calculation of the coefficients
$b_1(s)$ and $\widetilde{b}_1(s)$ explicitly. 

\begin{Prop}\label{sec-coeff}
Consider  the Laplace Beltrami operator $\Delta$ on the two-dimensional torus 
$\mathbb{T}^2$ and its spectral zeta function $\zeta(\Delta,s)$
The coefficients $b_0(s), \widetilde{b}_0(s), \Rel(s) \in (0,\alpha)$ in Theorem \ref{asymp-zeta} are 
given by\footnote{The extra summand in the expression for $V_2(s) b_1(s)$ is not directly a zeta function 
but is closely related. Such functions are known as angular lattice sums, see \cite{BGM}.}
\begin{align*}
V_2(s) \widetilde{b}_1(s) &= \frac{s\pi^2}{3} \zeta(\Delta,s-1), \\
V_2(s) b_1(s) &= \frac{s\pi^2}{3} \zeta(\Delta,s-1) + \frac{4\pi^2 V_2(s)}{(\alpha-s)}
\rint_0^\infty z^{4-2s+1} \sum_{k_1, k_2 \in \Z}
\frac{k_1^2k_2^2}{(k_1^2+k_2^2+z^2)^{2}}.
\end{align*}
\end{Prop}

\begin{proof} As asserted in \eqref{explicit-F}, we have
\begin{equation}\begin{split}
&F_{1,0} (x,y) =0, \ F_{1,1} (x,y) = \frac{2}{3} \pi^2 \Bigl( x^4 + y^4 \Bigr), \\
&\widetilde{F}_{1,0} (x,y) = 0, \ \widetilde{F}_{1,1} (x,y) = \frac{2}{3} \pi^2 \Bigl( x^2 + y^2 \Bigr)^2.
\end{split}\end{equation}
As in the proof of the previous result, let us write $I_{x_i}^{\R}$ and $E_{x_i,M}^{\R}$ for the operators in \eqref{IADE},
where the integration region is replaced by $\R$. Then, in view of \eqref{b0-theorem},
symmetry of the integrands around $x=0,y=0$ yields as in \eqref{RRR}
\begin{align*}
b_1(s) &= \frac{2}{3} \pi^2 \rint_0^\infty z^{2\alpha-2s-1} 
\left( E_{y,M}^{\R}\circ E_{x,M}^{\R}+2\cdot I_{y}^{\R} \circ E_{x,M}^{\R}\right)  \frac{x^4+y^4}{(x^2+y^2+z^2)^{3}} dz, \\
\widetilde{b}_1(s) &= \frac{2}{3} \pi^2 \rint_0^\infty z^{2\alpha-2s-1}
\left(E_{y,M}^{\R}\circ E_{x,M}^{\R}+2\cdot I_{y}^{\R} \circ E_{x,M}^{\R}\right)  \frac{(x^2+y^2)^2}{(x^2+y^2+z^2)^{3}} dz.
\end{align*}
We would like to apply the Euler-Maclaurin formula \eqref{e-emf-op} (with $n \to \infty$)
and argue precisely as in Proposition \ref{first-coeff} above.
However 
\begin{align*}
& \frac{F_{1,1} (x,y)}{(x^2+y^2+z^2)^{3}} =  \frac{2}{3} \pi^2 \frac{x^4+y^4}{(x^2+y^2+z^2)^{3}}, \\
& \frac{\widetilde{F}_{1,1} (x,y)}{(x^2+y^2+z^2)^{3}} =  \frac{2}{3} \pi^2  \frac{(x^2+y^2)^2}{(x^2+y^2+z^2)^{3}},
\end{align*}
are not $\R^2$ integrable and thus the error terms $A^n_x, A^n_y$ as well 
as $D^n_{x,M}, D^n_{y,M}$ (applied to the terms above) do not vanish in the
limit $n \to \infty$. Therefore we first perform an integration by parts trick
(for $b_1(s)$, the coefficient $\widetilde{b}_1(s)$ studied verbatim with 
$F_{1,1} (x,y)$ replaced by $\widetilde{F}_{1,1} (x,y)$)
\begin{align*}
b_1(s) &= + \frac{2}{3} \pi^2 \rint_0^\infty
\frac{z^{2\alpha-2s}}{2\alpha -2s}
\left( E_{y,M}^{\R}\circ E_{x,M}^{\R}+2\cdot I_{y}^{\R} \circ E_{x,M}^{\R}\right) \partial_z \frac{x^4+y^4}{(x^2+y^2+z^2)^{3}} dz \\
&+ \frac{2}{3} \pi^2  \ \underset{R\rightarrow\infty}{\mathrm{LIM}} \ 
\underset{\varepsilon \rightarrow 0}{\mathrm{LIM}} \ 
\frac{z^{2\alpha-2s}}{2\alpha -2s}
\left( E_{y,M}^{\R}\circ E_{x,M}^{\R}+2\cdot I_{y}^{\R} \circ E_{x,M}^{\R}\right) \frac{x^4+y^4}{(x^2+y^2+z^2)^{3}} \Big|^{z=R}_{z=\varepsilon.}
\end{align*}
For $\Rel(s) \in (0,\alpha/2)$, the regularized limit as $\varepsilon \to 0$ vanishes. For 
$M \in \N$ sufficiently large, the regularized limit as $R \to \infty$ vanishes as well. We conclude,
performing the same computation for $\widetilde{b}_1(s)$
\begin{align*}
b_1(s) &= \frac{2}{(\alpha-s)} \pi^2 \rint_0^\infty z^{2\alpha-2s+1} 
\left( E_{y,M}^{\R}\circ E_{x,M}^{\R}+2\cdot I_{y}^{\R} \circ E_{x,M}^{\R}\right)  \frac{x^4+y^4}{(x^2+y^2+z^2)^{4}} dz, \\
\widetilde{b}_1(s) &= \frac{2}{(\alpha - s)} \pi^2 \rint_0^\infty z^{2\alpha-2s+1}
\left(E_{y,M}^{\R}\circ E_{x,M}^{\R}+2\cdot I_{y}^{\R} \circ E_{x,M}^{\R}\right)  \frac{(x^2+y^2)^2}{(x^2+y^2+z^2)^{4}} dz.
\end{align*}
We can now apply the Euler-Maclaurin formula \eqref{e-emf-op} (with $n \to \infty$)
and conclude
\begin{align*}
b_1(s) &= \frac{2}{(\alpha-s)} \pi^2 \rint_0^\infty z^{2\alpha-2s+1} \sum_{k_1=-\infty}^\infty
\sum_{k_1=-\infty}^\infty \frac{k_1^4+k_2^4}{(k_1^2+k_2^2+z^2)^{4}} dz \\ &- 	
\frac{2}{3} \pi^2\rint_0^\infty z^{2\alpha-2s-1} \int_\R \int_\R \frac{x^4+y^4}{(x^2+y^2+z^2)^{4}} dx dy dz \\ 
\widetilde{b}_1(s) &= \frac{2}{(\alpha-s)} \pi^2 \rint_0^\infty z^{2\alpha-2s+1} \sum_{k_1=-\infty}^\infty
\sum_{k_1=-\infty}^\infty \frac{(k_1^2+k_2^2)^2}{(k_1^2+k_2^2+z^2)^{4}} dz \\ &- 	
\frac{2}{3} \pi^2 \rint_0^\infty z^{2\alpha-2s-1} \int_\R \int_\R \frac{(x^2+y^2)^2}{(x^2+y^2+z^2)^{4}} dx dy dz 
\end{align*}
The three-fold integrals vanish exactly as in \eqref{reg-zero} and hence we arrive at the
following intermediate formulae
\begin{align*}
b_1(s) &= \frac{2}{(\alpha-s)} \pi^2 \rint_0^\infty z^{2\alpha-2s+1} \sum_{k_1=-\infty}^\infty
\sum_{k_1=-\infty}^\infty \frac{k_1^4+k_2^4}{(k_1^2+k_2^2+z^2)^{4}} dz  \\ 
\widetilde{b}_1(s) &= \frac{2}{(\alpha-s)} \pi^2 \rint_0^\infty z^{2\alpha-2s+1} \sum_{k_1=-\infty}^\infty
\sum_{k_1=-\infty}^\infty \frac{(k_1^2+k_2^2)^2}{(k_1^2+k_2^2+z^2)^{4}} dz.
\end{align*}
To express the integrand in terms of resolvent traces we 
notice that the individual summands admit a partial fraction decomposition
\begin{equation}\label{bb}\begin{split}
\frac{k_1^4+k_2^4}{(k_1^2+k_2^2+z^2)^4}&=
 \frac{1}{(k_1^2+k_2^2+z^2)^2}- \frac{2 z^2}{(k_1^2+k_2^2+z^2)^3} \\ &+
 \frac{z^4}{(k_1^2+k_2^2+z^2)^4}- \frac{2 k_1^2k_2^2}{(k_1^2+k_2^2+z^2)^4}, \\
 \frac{(k_1^2+k_2^2)^2}{(k_1^2+k_2^2+z^2)^4} &= \frac{1}{(k_1^2+k_2^2+z^2)^2}-\frac{2z^2}{(k_1^2+k_2^2+z^2)^3}+\frac{z^4}{(k_1^2+k_2^2+z^2)^4}. 
\end{split} \end{equation}
Note the following identities that are special cases of  \eqref{e-zeta-reg-int}
\begin{equation}\label{zz}\begin{split}
&\rint_0^\infty z^{4-2(s-1)-1} \sum_{k_1=-\infty}^\infty
\sum_{k_1=-\infty}^\infty \frac{1}{(k_1^2+k_2^2+z^2)^{2}} = V^{-1}_2(s-1) \zeta(\Delta,s-1), \\
&\rint_0^\infty z^{6-2(s-1)-1} \sum_{k_1=-\infty}^\infty
\sum_{k_1=-\infty}^\infty \frac{1}{(k_1^2+k_2^2+z^2)^{3}} = V^{-1}_3(s-1) \zeta(\Delta,s-1), \\
&\rint_0^\infty z^{8-2(s-1)-1} \sum_{k_1=-\infty}^\infty
\sum_{k_1=-\infty}^\infty \frac{1}{(k_1^2+k_2^2+z^2)^{4}} = V^{-1}_4(s-1) \zeta(\Delta,s-1).
\end{split} \end{equation}
Combining \eqref{bb} and \eqref{zz}, we arrive at an expression of $b_1(s)$ and $\widetilde{b}_1(s)$
\begin{align*}
&\widetilde{b}_1(s) = \frac{2\pi^2}{(\alpha-s)} \zeta(\Delta,s-1)
\Bigg( V^{-1}_2(s-1) -2 V^{-1}_3(s-1) + V^{-1}_4(s-1)\Bigg), \\
&b_1(s) = \widetilde{b}_1(s) - \frac{4\pi^2}{(\alpha-s)}
\rint_0^\infty z^{4-2(s-1)-1} \sum_{k_1=-\infty}^\infty
\sum_{k_1=-\infty}^\infty \frac{k_1^2k_2^2}{(k_1^2+k_2^2+z^2)^{2}}.
\end{align*}
The extra factor in the expression for $\widetilde{b}_1(s)$ amounts using \eqref{VA} to
\begin{align*}
V_2^{-1}(s-1)-2V^{-1}_3(s-1)+V^{-1}_4(s-1) = \frac{\pi(2-s) }{2\sin(\pi s)}\frac{s(1-s)}{6}
= \frac{s(2-s)}{6} V^{-1}_2(s).
\end{align*}
From here the statement follows
\begin{align*}
\widetilde{b}_1(s) &= \frac{s\pi^2}{3} V^{-1}_2(s) \zeta(\Delta,s-1), \\
b_1(s) &= \frac{s\pi^2}{3} V^{-1}_2(s) \zeta(\Delta,s-1) \\ &+ \frac{4\pi^2}{(\alpha-s)}
\rint_0^\infty z^{4-2s+1} \sum_{k_1=-\infty}^\infty
\sum_{k_1=-\infty}^\infty \frac{k_1^2k_2^2}{(k_1^2+k_2^2+z^2)^{2}}.
\end{align*}
\end{proof}

\begin{Rem}	
This result makes the ultimate reason apparent, why we have introduced the
$9$-point star Laplacian $\widetilde{\Delta}$ and did not contend ourselves with the usual
$5$-point star Laplace operator $\Delta$. The reason is that for the former, the coefficient
$\widetilde{b}_1(s)$ can be expression completely in terms of the spectral 
zeta function.
\end{Rem}

We combine the results of Theorem \ref{asymp-zeta}, Propositions 
\ref{first-coeff} and \ref{sec-coeff}, and arrive at our first main result (see Theorem \ref{t-asymp})

\begin{Thm}\label{asymp-zeta2}
For $\Rel(s) \in (0, \alpha/2)$ we have for any integer $M \in \N$ as $n\to \infty$	
\begin{equation} \begin{split}
&\zeta(\Delta_n, s) = V_\alpha(s) \Bigg( a(s) \, n^{\alpha-2s}+\sum_{m=0}^{M-1} b_m(s) \, n^{-2m} \Bigg)
+ O(n^{-2M -2s +2}), \\
&\zeta(\widetilde{\Delta}_n, s) = V_\alpha(s) \Bigg( \widetilde{a}(s) \, n^{\alpha-2s}+\sum_{m=0}^{M-1} \widetilde{b}_m(s) \, n^{-2m} \Bigg)
+ O(n^{-2M -2s +2}).
\end{split}\end{equation}
The leading coefficients $a(s)$ and $\widetilde{a}(s)$ are explicitly given by 
\begin{equation}\begin{split}
a(s) & =  \int_0^\infty z^{2\alpha-2s-1}  \int_0^1 \int_0^1 
\Big(\frac{\sin^2(\pi x)}{\pi^2}+\frac{\sin^2(\pi y)}{\pi^2}+z^2\Big)^{-\alpha}  dx dy dz, \\
\widetilde{a}(s) &=  \int_0^\infty z^{2\alpha-2s-1}  \int_0^1 \int_0^1 
\Bigg( \frac{\sin^2(\pi x)}{\pi^2}+\frac{\sin^2(\pi y)}{\pi^2} + z^2\\  &-\frac{2 n^2}{3\pi^2}\sin^2\Big(\frac{\pi x}{n}\Big)\sin^2\Big(\frac{\pi y}{n}\Big) \Bigg)^{-\alpha}  dx dy dz.
\end{split}\end{equation}
The first two higher order coefficients are explicitly given by 	
\begin{equation}\begin{split}
b_0(s) &= \widetilde{b}_0(s) =  V_\alpha(s)^{-1} \zeta(\Delta,s), \quad
\widetilde{b}_1(s) = \frac{s\pi^2}{3} V_\alpha(s)^{-1} \zeta(\Delta,s-1), \\
b_1(s) &= \frac{s\pi^2}{3} V_\alpha(s)^{-1} \zeta(\Delta,s-1) + \frac{4\pi^2 V_2(s)}{(\alpha-s)}
\rint_0^\infty z^{4-2s+1} \! \! \!  \sum_{k_1, k_2 \in \Z}
\frac{k_1^2k_2^2}{(k_1^2+k_2^2+z^2)^{2}}.
\end{split}\end{equation}
\end{Thm}

\section{Epstein-Riemann conjecture and discrete zeta function}\label{ep-rie-hyp}
In \cite{FrKa} the authors showed that the discrete zeta function $\zeta(\mathcal{L}_n,s)$ on $\S^1_n$ 
has a connection to the Riemann conjecture, due to its asymptotic expansion as $n\rightarrow \infty$. 
Our goal is to prove an analogous result for our considerations of the Epstein zeta-function on a 
two-dimensional discrete torus. We use techniques of \cite{FrKa} and \cite{Fri} and define the function 
(using notation in Theorem \ref{asymp-zeta})
\begin{equation}\label{HH}\begin{split}
	H_n(s) &:= \pi^{-s}\Gamma(s)\Bigg(\zeta (\widetilde{\Delta}_n,s)-V_\alpha(s) \widetilde a(s)n^{2-2s}\Bigg)
	\\ &= \pi^{-s}\Gamma(s) \Bigg( \zeta(\Delta,s) + \frac{s\pi^2}{3} \zeta(\Delta,s-1) + O\Bigl(n^{-4}\Bigr) \Bigg).
\end{split}\end{equation}
where in the second equality we used the asymptotic expansion in Theorem \ref{asymp-zeta}
as $n\to \infty$ for $\Rel(s) \in (0,\alpha/2)$.
Let us introduce the following notation
\begin{equation}\label{complete-epstein-defn}
	\xi_2(s) := \pi^{-s} \, \Gamma(s) \, \zeta(\Delta,s), \quad 
	\Omega(s) :=\frac{1}{3} \, s \, \pi^{2-s} \, \Gamma(s) \, \zeta(\Delta,s-1).
\end{equation}
The function $\xi_2(s)$ is also known as the \emph{complete Epstein-zeta-function}. 
Using the new notation, we can write \eqref{HH} as
\begin{equation}\label{HHH}\begin{split}
	H_n(s) = \xi_2(s) + \Omega(s) n^{-2} + O\Bigl(n^{-4}\Bigr).
\end{split}\end{equation}

\noindent From here we directly obtain an asymptotic functional relation for $H_n(s)$.

\begin{Prop}\label{p-easy}
	For $\Rel(s)\in(0,\alpha/2)$ and $\zeta(\Delta,s) \neq 0$ 
	\begin{equation}
		\lim_{n\rightarrow \infty}\frac{H_n(1-s)}{H_n(s)}=1.
	\end{equation}
\end{Prop}
\begin{proof}
The statement follows directly from \eqref{HHH}, since (see e.g. \cite[(7)]{Eps} and \cite[p.59]{Ter}) the complete Epstein zeta function $\xi_{2}(s)$ admits the 
functional equation\footnote{This is also true for general dimensions $\alpha$ with $1$ replaced by $\frac{ \alpha}{2}$.}
\begin{equation}\label{e-ep-func-eq}
	\xi_{2}(s)=\xi_{2}(1-s), 
\end{equation}
\end{proof}

\noindent We conjecture, in full analogy to \cite[Sec. 9]{FrKa}, that the same asymptotic
functional relation holds (with absolute value) even at zeros of $\zeta(\Delta,s)$.
\begin{Conj}\label{disc-conj}
	For $s\in\C$ with $\Rel(s)\in(0,1)$ the following equality holds:
	\begin{equation}\label{e-disc-conj}
		\lim_{n\rightarrow \infty}\bigg|\frac{H_n(1-s)}{H_n(s)}\bigg|=1.
	\end{equation}
\end{Conj}

\subsection{Proof of Theorem \ref{ER-equiv}}
This section is devoted to the proof of our second main result, Theorem \ref{ER-equiv},
which we state here once again. 

\begin{Thm}\label{t-equiv}
Conjecture \ref{disc-conj} is equivalent to the Epstein-Riemann Conjecture \ref{h-eprie} for $\alpha = 2$,
i.e. to the claim that all non-trivial zeros of $\zeta(\Delta,s)$ have $\Rel(s) = \alpha/4 = 1/2$.
\end{Thm}

The proof follows by a sequence of lemmata. 

\begin{Lem}\label{l-half}
For $\Rel(s)=\frac 1 2$ we have
\begin{equation}\label{e-sec-term-eq}
\bigg|\frac{\Omega(1-s)}{\Omega(s)}\bigg|=1. 
\end{equation}
\end{Lem}

\begin{proof}
This is the same argument as in \cite[Lemma 15]{FrKa}.
Since $\zeta(\Delta, s)$ and $\Gamma(s)$ are real-valued on the real axis (away from poles), we conclude by the Schwarz reflection principle that
$\zeta(\Delta,\overline{s}) = \overline{\zeta(\Delta,s)}$ and $\Gamma(\overline{s}) = \overline{\Gamma(s)}$.
Thus the same holds for $\Omega(s)$ and we compute for $s=\frac 1 2 + i b, b \in \R$
	\begin{equation*}
		\Omega(1-s)=\Omega(1/2-ib)=\overline{\Omega(1/2+ib)} = \overline{\Omega(s)}.
	\end{equation*}
	This proves the statement.
\end{proof}

Lemma \ref{l-half} shows that \eqref{e-sec-term-eq} is true if $\Rel(s)=\frac{1}{2}$. 
We want to sharpen the statement and prove that \eqref{e-sec-term-eq} is satisfied only at $\Rel(s)=\frac{1}{2}$. 

\begin{Prop}\label{monotonicity}
For $\Rel(s) \in (0,1)$ with\footnote{The restriction to $\textup{Im}(s) > 65$ is technical, not conceptual. 
We simply did not manage to prove the various monotonicity statements in Lemma \ref{l-stric-inc} below 
without that restriction.} $\textup{Im}(s) > 65$, \eqref{e-sec-term-eq}
holds only at $\Rel(s)=\frac 1 2$.
\end{Prop}

\begin{proof}
In view of \eqref{complete-epstein-defn}, we can represent $\Omega(s)$ in terms of $\xi_2(s)$
\begin{equation}
	\Omega(s) = \frac{1}{3} \, s (s-1) \, \pi \,  \xi_2(s-1).
\end{equation}
From there we arrive at the following expression
\begin{equation}\label{omega1}
	\frac{\Omega(1-s)}{\Omega(s)}  = \frac{\xi_2(-s)}{\xi_2(s-1)} = \frac{\xi_2(s+1)}{\xi_2(s-1)}
	= \frac{s(s-1)}{\pi^2} \cdot \frac{\zeta(\Delta,s+1)}{\zeta(\Delta, s-1)},
\end{equation}
where we used the functional relation $\xi_{2}(s)=\xi_{2}(1-s)$ (recall \eqref{e-ep-func-eq}) in the 
nominator second equality, and plugged in the definition of $\xi_2(s)$ in the third equality. Similarly, we compute,
using \eqref{e-ep-func-eq} this time in the denominator
\begin{equation}\label{omega2}
	\frac{\Omega(1-s)}{\Omega(s)}  = \frac{\xi_2(-s)}{\xi_2(s-1)} = \frac{\xi_2(-s)}{\xi_2(2-s)}
	= \frac{ \pi^2}{ s(s-1)} \cdot \frac{ \zeta(\Delta,-s)}{\zeta(\Delta, 2-s)}.
\end{equation}
Let us introduce the following notation 
$$
 q(s) := \bigg| \frac{\pi^2}{s(s-1)} \bigg|, \quad 
 \eta(s):=\bigg|\frac{\zeta(\Delta,s+1)}{\zeta(\Delta,s-1)} \bigg|, \quad 
 \rho(s):=\bigg|\frac{\zeta(\Delta,2-s)}{\zeta(\Delta,-s)} \bigg|.
$$
The proof idea is as follows. We assume $\textup{Im}(s) > 65$ and prove that 
\begin{itemize}
\item $\bigg|\frac{\Omega(1-s)}{\Omega(s)}\bigg| = \bigg|\frac{q(s)}{\rho(s)}\bigg|$ is  is strictly increasing in $\Rel(s) \in (0,1/2)$,
\item $\bigg|\frac{\Omega(1-s)}{\Omega(s)}\bigg| = \bigg|\frac{\eta(s)}{q(s)}\bigg|$ is  is strictly increasing in $\Rel(s) \in (1/2,1)$.
\end{itemize}
Thus for any fixed $\textup{Im}(s) > 65$, $\bigg|\frac{\Omega(1-s)}{\Omega(s)}\bigg|$ is injective in $\Rel(s) \in (0,1)$. 
By Lemma \ref{l-half} it attains its value $1$ at $\Rel(s)=\frac 1 2$. By injectivity it attains this value only once. This proves the statement.
We establish the monotonicity claims in $2$ steps.  \bigskip

\noindent \underline{\textbf{Step 1: Monotonicity of $q(s)$}} \medskip

\noindent \textbf{Claim:} \emph{Provided $\mathrm{Im}(s)>1/4$,} 
\begin{itemize}
\item \emph{$q(s)$ is strictly increasing for $\Rel(s) \in (0,1/2)$,} 
\item \emph{$q(s)$ is strictly decreasing for $\Rel(s) \in (1/2,1)$.}
\end{itemize}

\noindent For $s\in\C$ we set $a=\Rel(s)$ and $b=\mathrm{Im}(s)$ and obtain
\begin{align*}
	q(a,t) = \frac{\pi^2}{\abs{(a+ib)(a+ib-1)}} 
\end{align*}
 For $\abs{b}>0$ fixed and $a\in(0,1)$ we calculate that 
\begin{align*}
	\partial_a q(a,b)^2= -\dfrac{\pi^4 \big( 4a^3-6a^2+2a+4b^2a-2b^2 \big)}{\left(a^4-2a^3+2b^2a^2+a^2-2b^2a+b^4+b^2\right)^2}. 
\end{align*}
First we have that $\partial_a q(1/2,b)^2=0$. Now for $\abs{b}>1/4$ the numerator of $\partial_a q(a,b)^2$ is strictly positive in
$a\in(0,1/2)$. Also for $\abs{b}>1/4$ the the numerator is strictly negative for $a\in(1/2,1)$. Thus, for $\abs{b}>1/4$,
$q(a,b)^2$ is strictly increasing for $a\in(0,1/2)$, strictly decreasing for $a\in(1/2,1)$ and has a maximum at $a=1/2$. \bigskip

\noindent \underline{\textbf{Step 2: Monotonicity of $\eta(s)$ and $\rho(s)$}} \medskip

\noindent \textbf{Claim:} \emph{Provided $\mathrm{Im}(s)>64$,} 
\begin{itemize}
\item \emph{$\rho(s)$ is strictly decreasing for $\Rel(s) \in (0,1/2)$,} 
\item \emph{$\eta(s)$ is strictly increasing for $\Rel(s) \in (1/2,1)$.}
\end{itemize}
This is in fact precisely the statement of Lemma \ref{l-stric-inc} below. \bigskip

\noindent \underline{\textbf{Concluding the proof}} \medskip

\noindent As a consequence, $| \Omega(1-s)/\Omega(s) |$ is strictly increasing 
(and thus injective) over the entire critical strip $\Rel(s)\in(0, 1)$. Thus, in view of Lemma \ref{l-half} 
we conclude for $\textup{Im}(s) > 65$ that
$| \Omega(1-s)/\Omega(s) | \neq 1$  for $\Rel(s) \neq \frac 1 2$. This proves the statement.
\end{proof}

Now we are able to prove the Theorem \ref{t-equiv}. 
\begin{proof}[Proof of Theorem \ref{t-equiv}]
First, we assume the E.R. Conjecture \ref{h-eprie} for $\alpha = 2$ is true. We want to conclude
that Conjecture \ref{disc-conj} holds. By Proposition \ref{p-easy} we only need to consider $\Rel(s) \in (0,1)$ 
where $\xi_2(s)=0$. For such $s$ we find by \eqref{HHH} 
$$
\lim_{n\rightarrow \infty}\bigg|\frac{H_n(1-s)}{H_n(s)}\bigg| = \bigg|\frac{\Omega(1-s)}{\Omega(s)}\bigg|.
$$
By assumption, all non-trivial zeros of $\xi_{2}(s)$ have $\Rel(s)=\frac 1 2$. In that case 
the right hand side above equals $1$ by Lemma \ref{l-half}. This proves the Conjecture \ref{disc-conj}.\medskip
		
Conversely, assume that Conjecture \ref{disc-conj} holds. We want to conclude that the
E.R. Conjecture \ref{h-eprie} for $\alpha = 2$ is true. It is known since \cite{LRW} and many more, that the E.R. conjecture \ref{h-eprie} holds for $\mathrm{Im}(s)\leqslant 65$. Now consider a zero $s$ of 
$\zeta(\Delta,s)$ with $\mathrm{Im}(s)> 65$. We compute by assumption
$$
1= \lim_{n\rightarrow \infty}\bigg|\frac{H_n(1-s)}{H_n(s)}\bigg| = \bigg|\frac{\Omega(1-s)}{\Omega(s)}\bigg|.
$$
Thus \eqref{e-sec-term-eq} holds and hence by Proposition \ref{monotonicity} we conclude $\Rel(s) = \frac 1 2$.
This proves the E.R. Conjecture \ref{h-eprie} for $\alpha = 2$.
\end{proof}

\subsection{Proof of auxiliary estimates}
Note first the following identity, cf. \cite{SrZv}.
\begin{Lem}\label{l-mon}
	Let $G\subset\C$ be any domain and $f:G\rightarrow\C$ a holomorphic function with $f(z)\neq0$. 
	Then for any $z\in G$ with $z=a+ib$ we have
	\begin{equation}
		\Rel\left(\frac{f'(z)}{f(z)}\right)=\frac{1}{\abs{f(z)}}\frac{\partial\abs{f(z)}}{\partial a}.
	\end{equation}
\end{Lem}

\noindent This Lemma will be a main tool in the proof of the following statement. 

\begin{Lem}\label{l-stric-inc}
Let $s\in\C$ with $\mathrm{Im}(s)>65$. Then 
\begin{itemize}
\item \emph{$\rho(s) = \bigg|\cfrac{\zeta(\Delta,2-s)}{\zeta(\Delta,-s)}\bigg|$ is strictly decreasing for $\Rel(s) \in (0,1/2)$,} 
\item \emph{$\eta(s) = \bigg|\cfrac{\zeta(\Delta,s+1)}{\zeta(\Delta,s-1)}\bigg|$ is strictly increasing for $\Rel(s) \in (1/2,1)$.}
\end{itemize}
\end{Lem}
\begin{proof}
	Let us first show the second statement. Let $s=a+ib\in\C$ with $a\in (1/2,1)$. The restriction $b= \mathrm{Im}(s)>65$
	comes only at the very last step of the proof. By Lemma \ref{l-mon}, showing  that
	\begin{equation*}
	\eta(s)=\bigg|\frac{\zeta(\Delta,s+1)}{\zeta(\Delta, s-1)}\bigg|
	\end{equation*}
	is strictly increasing in $a$, is equivalent to 
		\begin{equation}\label{equiv1}
		\Rel\left(\frac{\zeta'(\Delta, s+1)}{\zeta(\Delta,s+1)}\right)-\Rel\left(\frac{\zeta'(\Delta,s-1)}{\zeta(\Delta,s-1)}\right)>0.
	\end{equation}
	By \eqref{e-ep-rie-beta}, this inequality is equivalent to 
	\begin{equation*}
		\Rel\left(\frac{\zeta'_R(s+1)}{\zeta_R(s+1)}\right)-\Rel\left(\frac{\zeta_R'(s-1)}{\zeta_R(s-1)}\right)	
		+\Rel\left(\frac{\beta'(s+1)}{\beta(s+1)}\right)-\Rel\left(\frac{\beta'(s-1)}{\beta(s-1)}\right)>0.
	\end{equation*}
	Let us consider the individual terms on the left hand side. First we note for $\Lambda(n)$ being the Mangoldt function\footnote{We
	use the relation between Dirichlet series $D(s) = \sum_{n=1}^\infty \frac{\ell(n)}{n^s}$ for a completely multiplicative function $\ell(n)$ 
	and the Mandgoldt function $\Lambda(n)$, whenever the series converge:
	$$D'(s) / D(s) =  -\sum_{n=1}^\infty \frac{\ell(n)\Lambda(n)}{n^s}.$$} and $a>1/2$
	\begin{align*}
		\Rel\left(\frac{\zeta_R'(s+1)}{\zeta_R(s+1)}\right)=-\Rel\left(\sum_{n\in\N} \Lambda(n)n^{-s-1}\right)=-\sum_{n\in\N} \Lambda(n)n^{-a-1}\cos(b\log(n)).
	\end{align*}
	Then the absolute value may be estimated in the following way 
	\begin{equation*}
		\abs{\Rel\left(\frac{\zeta'_R(s+1)}{\zeta_R(s+1)}\right)}\leqslant\sum_{n\in\N} \Lambda(n)n^{-\frac 3 2}\leqslant1.51
	\end{equation*}
	A similar estimate holds for the Dirichlet beta function for $a>1/2$
	\begin{equation*}
		\Rel\left(\frac{\beta'(s+1)}{\beta(s+1)}\right)>-1.51.
	\end{equation*}
	So the second statement (equivalently \eqref{equiv1}) follows if (for $\abs{b}>65$)
	\begin{equation}\label{equiv2}
		-\Rel\left(\frac{\zeta'_R(s-1)}{\zeta_R(s-1)}\right)-\Rel\left(\frac{\beta'(s-1)}{\beta (s-1)}\right)>3.02. 
	\end{equation}
	To obtain such an estimate, we use techniques from \cite{MSZ}. 
	First we consider the \emph{complete Riemann zeta function} $\widetilde\xi(s)$ and the
	\emph{completed Dirichlet beta function} $\xi(s,\chi)$ introduced by \cite{Fri} for any Dirichlet 
	$L$-function associated to a Dirichlet character $\chi$
	\begin{align*}
	&\widetilde\xi(s)=(s-1)\Gamma\left(\frac s 2+1\right)\pi^{s/2}\zeta_R(s), \\
	&\xi(s,\chi)=\left(\frac \pi 4\right)^{-s/2}\Gamma\left(\frac s 2\right)\beta(s),
	\end{align*}
	where the factor in front of $\zeta_R(s)$ cancels the pole at $s=1$.
	We obtain by straightforward computations
\begin{align*}
&0 > \Rel\left(\frac{\tilde\xi'(s)}{\tilde\xi(s)}\right) =\Rel\left(\frac{1}{s-1}\right)+\frac 1 2\Rel\left(\psi\left(\frac {s+1}{ 2}\right)\right)
-\frac{\log(\pi)}{2}+\Rel\left(\frac{\zeta'_R(s)}{\zeta_R(s)}\right), \\
&0 >\Rel\left(\frac{\xi'(s,\chi)}{\xi(s,\chi)}\right) = \frac 1 2\log\left(\frac 4 \pi\right)+\frac 1 2 
\Rel\left(\psi\left(\frac{s-1}{2}\right)\right) + \Rel\left(\frac{\beta'(s-1)}{\beta (s-1)}\right),
\end{align*}	
where the inequalities hold for $a<0$ by \cite[Corollary 2.6 and Corollary 2.6 L]{MSZ}.
Here, we have introduced the digamma-function $\psi(s)=\frac{\Gamma'(s)}{\Gamma(s)}$.  
Since for $\Rel(s) < 1$, the real part of $(s-1)$ is negative, we conclude 
	\begin{equation}\label{zeta-beta}\begin{split}
		&-\Rel\left(\frac{\zeta'_R(s-1)}{\zeta_R(s-1)}\right)>\Rel\left(\frac{1}{s-2}\right)+\frac 1 2\Rel\left(\psi\left(\frac{ s+1} {2}\right)\right)-\frac{\log(\pi)}{2},\\
		&-\Rel\left(\frac{\beta'(s-1)}{\beta (s-1)}\right)>\frac 1 2\log\left(\frac 4 \pi\right)+\frac 1 2 \Rel\left(\psi\left(\frac{s-1}{2}\right)\right). 
	\end{split}\end{equation}
	Let us look at the individual terms on the right hand sides of the inequalities above. 
	According to \cite[Lemma 3.2 (iii)]{MSZ} (a special case of the Stirling series for the digamma function) for a $0<\theta<\pi$
	and $s \in \C$ in the sector $-\theta<\arg(s)<\theta$ we have that 
		\begin{equation}\label{di-gamma-est}
		\Rel\left(\psi\left(s\right)\right)=\log\left(\abs{s}\right)-\frac{\Rel(s)}{2\abs{s}^2}+\Rel\left(R'_0\left(s\right)\right),
	\end{equation}
	where $R'_0(s)$ is a function, satisfying the estimate $\abs{R'_0\left(\frac{s+1}{2}\right)}\leqslant \frac{\sec^3(\theta/2)}{12\abs{s}^2}$. \medskip
	
	We choose the angle $\theta=1.02452 \, \pi / 2$ such that $(s-1)/2$ and $(s+1)/2$, for $a\in (1/2,1)$ and $\abs{b} = \abs{\textup{Im}(s)} >65$, are in the sector.
	Then for this specific choice of $\theta$ we have that $\sec^3(\theta/2)<3$. Thus \eqref{di-gamma-est} yields for $s=a+ib\in\C$ with 
	$a \in (1/2,1)$ and $\abs{b} > 65$
	\begin{equation}\label{digamma1}\begin{split}
	\Rel\left(\psi\left(\frac{s+1}{2}\right)\right)\geqslant\log\left(\frac{\abs{b}}{2}\right) -\frac{2}{b^2}, \\
	\Rel\left(\psi\left(\frac{s-1}{2}\right)\right)\geqslant\log\left(\frac{\abs{b}}{2}\right)  -\frac{2}{b^2}. 
	\end{split}\end{equation}
	For the remaining term in \eqref{zeta-beta} we find
	\begin{equation}\label{rest1}\begin{split}
	&\Rel\left(\frac{1}{s-2}\right)=\frac{a-2}{(a-2)^2+b^2} > -\frac{2}{1+b^2}.	
	\end{split}\end{equation}
	Plugging \eqref{digamma1} and \eqref{rest1} into \eqref{zeta-beta}, we find
		\begin{align*}
		&-\Rel\left(\frac{\zeta'_R(s-1)}{\zeta_R(s-1)}\right)-\Rel\left(\frac{\beta'(s-1)}{\beta (s-1)}\right) 
		\\ &>\log\left(\frac{\abs{b}}{2}\right) + \frac 1 2\log\left(\frac 4 \pi\right)-\frac{\log(\pi)}{2}-\frac{2}{1+b^2} -\frac{2}{b^2}>3.02,
	\end{align*}
	where only in the last inequality we assumed $\abs{b}>65$ to obtain $3.02$ as a lower bound. 
	This proves \eqref{equiv2} and thus the second statement. \medskip
	
	\noindent To prove the first statement, we consider
		\begin{equation*}
		\rho(s) = \bigg|\frac{\zeta(\Delta,2-s)}{\zeta(\Delta,-s)}\bigg|:\left\{\Rel(s) \in (0,1/2):\abs{b}>65\right\}\rightarrow\R
	\end{equation*} 
	as a composition of $\eta(s)$ and 
	\begin{equation*}
		g(s)=1-s:\left\{\Rel(s) \in (0,1/2):\abs{b}>65\right\}\rightarrow\left\{\Rel(s) \in (1/2,1):\abs{b}>65\right\}, 
	\end{equation*}
	such that $\rho(s) = (\eta \circ g)(s)$. Let us rewrite this as
	$$
	\big|\rho(s)\big| \equiv \big|\eta \big| \bigg(\Rel(g(s)), \textup{Im} (g(s))\bigg).
	$$
	By the proof of the second statement, $|\eta(s)| = |\eta|(a,b)$ is increasing in the first variable
	$a \in  (1/2,1)$ for any fixed $\abs{b}> 65$. Since $\Rel (g(s))$ is 
	strictly decreasing in $\Rel(s) \in (0,1/2)$, it follows immediately that $ \big|\eta \big| \bigg(\Rel(g(s)), \textup{Im} (g(s))\bigg)$ is too. 
        This proves the first statement.
	\end{proof}

\section{Open problems and future research directions}\label{open-prob}

Our arguments apply to tori of higher dimension $\alpha \geqslant 3$ as well. By choosing an 
appropriate refinement of the discrete Laplacian (as we have done in two dimensions by studying
the $9$-point star Laplacian instead of the classical $5$-point star operator), we may obtain a 
similar functional relation as in Conjecture \ref{disc-conj}
in higher dimensions as well. \medskip

On the other hand, as pointed out in Answer \ref{answer}, the Epstein-Riemann conjecture does not hold
for dimensions $\alpha \geqslant 3$. We hope that by studying the discrete functional relation, we may 
find a discrete geometric reason why the Epstein-Riemann conjecture fails in higher dimensions. \medskip

Finally, an interesting question is if our main result, Theorem \ref{t-asymp}, can be generalized 
to other geometries in the spirit of \cite{Iz} beyond tori.


\begin{thebibliography}{99}

\bibitem[\textsc{BGM13}]{BGM}
 J. M. Borwein, M. L. Glasser, R. C. McPhedran, J. G. Wan and I. J. Zucker,
 \textit{Lattice Sums Then and Now},
 Cambridge University Press, 2013, New York.

\bibitem[\textsc{Bra13}]{Bra}
D. Braess, 
\textit{Finite Elemente},
Springer Spektrum,  2013, Berlin, Volume 5

\bibitem[\textsc{CJK10}] {CJK}
G. Chinta, J. Jorgensen and A. Karlsson,
\textit{Zeta functions, heat kernel and spectral asymptotics on degenerating families of discrete tori},
Nagoya Mathematical Journal, 2010, Volume 198, Issue 5, pp. 121-172.

\bibitem[\textsc{DuDa88}] {DuDa88}
B. Duplantier and F. David, 
\textit{Exact partition functions and correlation functions
of multiple Hamiltonian walks on the Manhattan lattice}, Journal of Statistical Physics, 1988, Volume 51, Issue 3, pp. 327-434.
no. 3-4, 327-434.

\bibitem[\textsc{Edw01}]{Edw}
H. Edwards, \textit{Riemann zeta function},
Dover Publications, 2001, Mineola.

\bibitem[\textsc{Eps03}]{Eps}
P. Epstein, \textit{Zur Theorie allgemeiner Zetafunctionen},
Mathematische Annalen, 1903, Volume 56, Issue 4, pp. 615-644.

\bibitem[\textsc{Fin20}]{Finski}
S. Finski, \emph{Spanning trees, cycle-rooted spanning forests on discretizations of flat surfaces and analytic torsion}, \url{	arXiv:2001.05162}, 2020.

\bibitem[\textsc{Fri16}] {Fri}
F. Friedli, 
\textit{A functional relation for L-functions of graphs equivalent to the Riemann conjecture for Dirichlet L-functions},
Journal of Number Theory, 2016, Volume 169, Issue 20, pp. 342-352.

\bibitem[\textsc{FrKa17}] {FrKa}
F. Friedli and A. Karlsson,  
\textit{Spectral zeta functions of graphs and the Riemann zeta function in the critical strip},
Tohoku Mathematical Journal, 2017, Volume 69, Issue 4, pp. 585-610.


\bibitem[\textsc{Haw77}]{Haw}
S. W. Hawking,
\textit{Zeta function regularization of path integrals in curved spacetime},
Communications in Mathematical Physics, 1977, Volume 55, Issue 2, pp. 133-148.

\bibitem[\textsc{IzKh20}]{Iz}
K. Izyurov and M. Khristoforov, \emph{Asymptotics of the determinant of discrete Laplacians on triangulated and quadrangulated surfaces}, \url{arXiv:2007.08941}, 2020.


\bibitem[\textsc{Kaw20}]{Kaw}
A. Kawalec, 
\textit{On the complex magnitude of Dirichlet beta function},
Computational Methods in Science and Technology, 2020, Volume 26, Issue 1, pp. 21-28.


\bibitem[\textsc{Ken00}]{Ken}
R. Kenyon,
\textit{The asymptotic determinant of the discrete Laplacian},
Acta Mathematica, 2000, Volume 185, Issue 2, pp. 239-286. 

\bibitem[\textsc{Les96}]{Les}
M. Lesch, \textit{Differential Operators of Fuchs Type, Conical Singularities and Asymptotic Methods},
\url{arXiv:dg-ga/9607005}, 1996.

\bibitem[\textsc{Les98}]{Les:DOR}
\bysame, \textit{Determinants of regular singular {S}turm-{L}iouville operators},
 Mathematische Nachrichten, 1998, Volume 194, Issue 1, pp. 139-170.


\bibitem[\textsc{LeVe13}]{LesVer}
M. Lesch and B. Vertman, \textit{Regularizing infinite sums of
  zeta-determinants}, 2015, Mathematische Annalen. Volume 361, Issue 3, pp. 835-862.
  
\bibitem[\textsc{Lov12}]{Lo}
L. Lov\'asz, \textit{Large networks and graph limits}, American Mathematical Society Colloquium Publications, 2012, Volume 60, American Mathematical Society, Providence. 
\bibitem[\textsc{LRW86} ]{LRW}
J. van de Lune, H. J. J. te Riele, and D. T. Winter,
\textit{On the Zeros of the Riemann Zeta Function in the Critical Strip. IV},
Mathematics of Computation, 1986, Volume 46, Issue 174, pp. 667-681.

\bibitem[\textsc{LyMc69}] {LyMH}
 J. N. Lyness and J. B. B. McHugh, 
\textit{On the remainder term in the N-dimensional Euler Maclaurin expansion},
Numerische Mathematik, 1970, Volume 15, Issue 4, pp. 333-344.

\bibitem[\textsc{MiPl49}]{MiPl}
S. Minakshisundaram and A. Pleijel,
\textit{Some Properties of the Eigenfunctions of The Laplace-Operator on Riemannian Manifolds},
Canadian Journal of Mathematics, 1949, Volume 1, Issue 3, pp. 242-256.

\bibitem[\textsc{MoLy98}]{MoLy}
G. Monegato and J. N. Lyness.
\textit{The Euler-Maclaurin expansion and finite-part integrals},
Numerische Mathematik, 1998, Volume 81, Issue 2, pp. 273-291.

\bibitem[\textsc{MSZ14}]{MSZ}
Y. Matiyasevich, F. Saidak and P. Zvengrowski,
\textit{Horizontal monotonicity of the modulus of the zeta function, L-function and related functions},
Acta Arithmetica, 2014, Volume 166, Issue 2, pp. 189-200.

\bibitem[\textsc{McPhe13}]{Phe}
R. C. McPhedran, \emph{The Riemann Hypothesis for Dirichlet L-Functions}, \url{arXiv:1308.6431}, 2013.
	
\bibitem[\textsc{ReVe15}]{RV}
N. Reshetikhin and B. Vertman, \emph{Combinatorial quantum field theory and gluing
formula for determinants}, Letters in Mathematical Physics, 2015, Volume 105, Issue 3, pp. 309-340.

\bibitem[\textsc{See67}]{Se}
R. T. Seeley, 
\textit{Complex powers of an elliptic operator},
Singular Integrals, 1967, Issue 17, pp. 288-307.

\bibitem[\textsc{Sid04}]{Sid}
A. Sidi,
\textit{Euler-Maclaurin expansions for integrals with endpoint singularities: a new perspective},
Numerische Mathematik, 2004, Volume 98, Issue 2, pp. 371-387.

\bibitem[\textsc{Sri15}]{Sridar} 
A. Sridhar, \emph{Asymptotic Determinant of Discrete Laplace-Beltrami Operators}, \url{arXiv:1501.02057}, 2015.

\bibitem[\textsc{SrZv11}]{SrZv}
G. K. Srinivasan and P. Zvengrowski,
\textit{On the Horizontal Monotonicity of $|\Gamma(s)|$},
Canadian Mathematical Bulletin, 2011, Volume 54, Issue 3, pp. 538-543.

\bibitem[\textsc{Ter85}]{Ter} A. Terras 
\textit{Harmonic Analysis on Symmetric Spaces and Applications I}, 
Springer-Verlag, 1985, Volume 1,  New York.

\bibitem[\textsc{TrSa19}]{TrSa}
I. Travenec and L. Samaj
\textit{Generation of off-critical zeros for hypercubic Epstein zeta-functions},
\url{arXiv:1909.07112}, 2019.

\bibitem[\textsc{VPL97}]{VPV}
P. Verlinden, D. Potts and J. N. Lyness,
\textit{Error expansions for multidimensional trapezoidal rules with Sidi transformations}, Numerical Algorithms. 1997, Volume 16, Issue 3, pp. 321-347.

\bibitem[\textsc{Ver17}] {Ver}
B. Vertman,
\textit{Regularized limit of determinants for discrete tori},
Monatshefte f\"ur Mathematik, 2017, Volume 186, Issue 3, pp. 539-557.

\bibitem[\textsc{Zuc74}]{Zuc}
I. J. Zucker,
\textit{Exact results for some lattice sums in 2, 4, 6 and 8 dimensions},
Journal of Physics A: Mathematical, Nuclear and General, 1974, Volume 7, Issue 13, pp. 1568-1575.

\bibitem[\textsc{ZMGR14}]{ZMGR}
D. Zwillinger, V. Moll, I.S. Gradshteyn and I.M. Ryzhik,
\textit{Table of Integrals, Series, and Products},
Academic Press, 2007, Volume 7.

\end{thebibliography}
\end{document}